\documentclass[letterpaper,11pt]{amsart}
\usepackage{indentfirst} 
\usepackage{amssymb}
\usepackage{mathrsfs}
\usepackage{amsmath}
\usepackage{amsthm}
\usepackage[colorlinks = true, linkcolor = black, citecolor = blue ]{hyperref}
\usepackage{thmtools}
\usepackage{enumerate}
\usepackage[left=3cm, right=3cm, bottom=3.4cm]{geometry}
\usepackage{bbm}             
\usepackage[usenames,dvipsnames]{xcolor} 
\usepackage{tikz}

\theoremstyle{plain}
\newtheorem{theorem}{Theorem}[section]

\theoremstyle{plain}
\newtheorem{proposition}{Proposition}[subsection]
\newtheorem{lemma}[proposition]{Lemma}
\newtheorem{corollary}[proposition]{Corollary}

\theoremstyle{definition}
\newtheorem{definition}{Definition}[subsection]

\theoremstyle{remark}
\newtheorem{remark}{Remark}[subsection]

\declaretheorem[name=Acknowledgements,numbered=no]{ack}

\newcommand{\vertiii}[1]{{\left\vert\kern-0.25ex\left\vert\kern-0.25ex\left\vert #1 
    \right\vert\kern-0.25ex\right\vert\kern-0.25ex\right\vert}}

\DeclareMathOperator\vol{vol}
\DeclareMathOperator\dvol{dvol}
\DeclareMathOperator\supp{supp}
\DeclareMathOperator\spann{span}
\def\e{\epsilon}
\def\R{{\mathbb R}}

\def\N{{\mathbb N}}

\def\H{{\mathbb H}}
\def\D{{\mathcal D}}

\def\S{{\mathbb S}}
\def\diam{\mbox{\rm diam} }

\def\Hor{\mbox{Hor}}
\def\Ver{\mbox{Ver}}

\def\M{\mathcal{M}}
\def\max{\mathrm{max}}

\begin{document}

\title[Decay properties of Vlasov fields on asymptotically hyperbolic manifolds]{Decay properties of Vlasov fields on non-trapping asymptotically hyperbolic manifolds}

\author[Anibal Velozo Ruiz]{Anibal Velozo Ruiz} \address{Facultad de matem\'atica, Pontificia Universidad Cat\'olica de Chile, Avenida Vicu\~na Mackenna 4860, Santiago, Chile.}
\email{apvelozo@mat.uc.cl}

\author[Renato Velozo Ruiz]{Renato Velozo Ruiz} \address{Laboratory Jacques-Louis Lions (LJLL), University Pierre and Marie Curie (Paris 6), 4
place Jussieu, 75252 Paris, France.}\address{Department of Mathematics, University of Toronto, 40 St. George Street, Toronto, ON, Canada.}
\email{renato.velozo.ruiz@utoronto.ca}

\begin{abstract}
In this paper, we study pointwise decay estimates in time for Vlasov fields on non-trapping asymptotically hyperbolic manifolds. We prove optimal decay estimates in time for the spatial density induced by Vlasov fields on these geometric backgrounds in dimension two. First, we show exponential decay for Vlasov fields on hyperbolic space supported away from the zero velocity set. In contrast, we obtain inverse polynomial decay for general Vlasov fields on hyperbolic space. In the second part of the article, we prove exponential decay for Vlasov fields on non-trapping asymptotically hyperbolic manifolds supported away from the zero velocity set. The proofs are obtained through a commuting vector field approach. We exploit the hyperbolicity of the geodesic flow in these geometric backgrounds, by making use of a commuting vector field in the unstable invariant distribution of phase space. 
\end{abstract}


\maketitle

\setcounter{tocdepth}{1}
\tableofcontents

\section{Introduction}\label{section_introduction_hyperbolic}

In this paper, we study the evolution in time of collisionless many-particle systems on a Riemannian manifold $(\M,g)$. We consider collisionless many-particle systems described statistically by a distribution function satisfying a transport equation on phase space. More precisely, we investigate the linear dynamics of the solutions $f: [0,\infty) \times T\M\to [0,\infty)$ to the \emph{Vlasov equation on a Riemannian manifold} $(\M,g)$, given by $$\partial_tf+Xf=0,$$ in terms of the generator of the geodesic flow $X\in TT\M$ of the Riemannian manifold $(\M,g)$. The Vlasov equation, also known as the Liouville equation, is motivated by classical statistical mechanics. See \cite{T79, LP81} for further details.

The Vlasov equation on a Riemannian manifold describes the evolution in time of a collisionless system whose particles follow the trajectories set by the geodesic flow in the geometric background. The trajectories determined by the geodesic flow describe the motion of free falling particles on a Riemannian manifold. The Vlasov equation on a Riemannian manifold is motivated by its fundamental role in non-linear kinetic PDE systems on Riemannian manifolds. See \cite{G96, CK02} for more information on PDE models arising in kinetic theory. Consider, for instance, the Vlasov--Poisson system on a Riemannian manifold $(\M,g)$. This kinetic model describes a collisionless system on $(\M,g)$, for which the trajectories described by its particles are determined by the generator of the geodesic flow in $(\M,g)$, and the mean field generated by the many-particle system. The Vlasov--Poisson system on Euclidean space $(\R^n_x,\delta_E)$ has been extensively studied in the scientific literature due to its physical relevance \cite{LP81, BT11}. See \cite{DILS16} for more information about the Vlasov--Poisson system on hyperbolic space $\H^2$, and the unit sphere $\S^2$.

The Vlasov equation on a Riemannian manifold $(\M,g)$ is a linear transport equation along the geodesic flow in $(\M,g)$. Naturally, the dynamics described by Vlasov fields on a Riemannian manifold $(\M,g)$ depend strongly on the particular form of the geodesic flow in the phase space of the background. We consider a suitable class of \emph{non-trapping asymptotically hyperbolic geometric backgrounds} for which the corresponding geodesics escape to infinity. In this article, we are specifically interested on geometric collisionless systems that are \emph{dispersive}, in the sense that decay estimates in time hold for the \emph{spatial density} induced by Vlasov fields $f(t,x,v)$ on a Riemannian manifold $(\M,g)$. We define the \emph{spatial density} on a Riemannian manifold $(\M,g)$, by $$\rho(f)(t,x):=\int f(t,x,v)\dvol_{T_x\M}(v),$$ in terms of the corresponding volume form on $T\M$ with respect to the corresponding Riemannian metric. 

Classical dispersive collisionless systems are given by Vlasov fields on Euclidean space $(\R^n,\delta_{E})$. In the PDE literature, the first decay estimates for Vlasov fields on Euclidean space were obtained by Bardos and Degond \cite{BD}, who proved that the spatial density decays inverse polynomially in time for compactly supported initial data. The decay estimates for the spatial density in \cite{BD} use the method of characteristics which exploits the explicit representation of the geodesic flow in Euclidean space. Later, Strichartz estimates were obtained for Vlasov fields on Euclidean space by Castella and Perthame \cite{CasPer}. More recently, new decay estimates for Vlasov fields on Euclidean space were obtained by Smulevici \cite{Sm}, who proved that the spatial density decays inverse polynomially in space and time without assuming compact support on the initial distribution. The decay estimates for the spatial density in \cite{Sm} use a robust vector field method that exploits commuting vector fields for the Vlasov equation on Euclidean space. The methods developed in \cite{BD} and \cite{Sm} to obtain decay estimates for the spatial density on Euclidean space, were used in these works to prove the \emph{non-linear stability of the vacuum solution} for the Vlasov--Poisson system on Euclidean space. The \emph{vacuum solution} for the Vlasov--Poisson system on Euclidean space is defined as the distribution $f\equiv 0$ that vanishes everywhere. 

Vector field methods have been used to obtain decay estimates for collisionless systems in several settings. In \cite{FJS17}, Fajman, Joudioux, and Smulevici, developed a vector field method to prove decay estimates in space and time for velocity averages induced by \emph{relativistic Vlasov fields} on Minkowski spacetime. The methods developed in \cite{FJS17} were used later to prove the non-linear stability of Minkowski spacetime, as a solution of the Einstein--massive Vlasov system by Fajman et al. \cite{FJS}, and as a solution of the Einstein--massless Vlasov system by Bigorgne et al. \cite{BFJST}. Around the same time, and independently, vector fields methods were used to prove the non-linear stability of Minkowski spacetime, as a solution of the Einstein--massless Vlasov system by Taylor \cite{T}, and as a solution of the Einstein--massive Vlasov system by Lindblad and Taylor \cite{LT}. We observe that the article by Taylor \cite{T} underscores the relevance of Jacobi fields in the tangent bundle of spacetime for the use of vector field methods in geometric backgrounds. 

In this paper, we study the Vlasov equation on non-compact Riemannian manifolds, for which the associated geodesic flow is \emph{hyperbolic}. The geodesic flow in a Riemannian manifold is \emph{hyperbolic}, if there exists an invariant decomposition of the tangent of the unit tangent bundle, into three subspaces: the subspace $E_0$ spanned by the generator of the geodesic flow, a \emph{stable subspace} $E_s$ where the differential of the geodesic flow contracts uniformly, and an \emph{unstable subspace} $E_u$ where the differential of the geodesic flow expands uniformly. Anosov \cite{ano} proved that the class of Riemannian manifolds with bounded and strictly negative sectional curvature have hyperbolic geodesic flows. See \cite{KH, Pat99} for further details on hyperbolic geodesic flows. In particular, \emph{hyperbolic space} $\H^n$ -- the unique simply connected Riemannian manifold with constant sectional curvature equal to minus one -- has a hyperbolic geodesic flow. The Vlasov equation on compact negatively curved Riemannian manifolds is a classical subject of study in hyperbolic dynamical systems, motivated by the rich chaotic dynamics determined by the corresponding geodesic flows. We, on the other hand, are interested in studying the Vlasov equation on Riemannian manifolds where the corresponding geodesic flow disperse.

In this article, we prove pointwise decay in time for the spatial density induced by a Vlasov field on non-trapping asymptotically hyperbolic manifolds. Decay estimates for the spatial density on these geometric backgrounds hold due to the lack of recurrence of the corresponding geodesic flows. We establish pointwise decay estimates for the spatial density by proving time decay for the volume of the velocity support of the distribution function. The proofs of our main results exploit the hyperbolicity of the geodesic flow to estimate the derivatives of the spatial density. We make use of commuting vector fields contained in suitable distributions of phase space\footnote{We call a \emph{distribution} $\Delta$ in $T\M$ to a regular map $(x,v)\mapsto \Delta_{(x,v)}\subseteq T_{(x,v)}T\M$, where $\Delta_{(x,v)}$ are vector subspaces satisfying suitable conditions (in the standard sense in differential geometry).}. The commuting vector fields used in this paper are Jacobi fields in the tangent bundle with respect to the Sasaki metric. 

The geodesics on non-trapping asymptotically hyperbolic manifolds escape to infinity, thus, the main part of the analysis of Vlasov fields on these backgrounds is carried out in the far-away region. We consider asymptotically hyperbolic manifolds, so the geometry in the far-away region is close to the one in hyperbolic space. With this motivation, we first prove pointwise decay in time for the spatial density induced by Vlasov fields on hyperbolic space. The hyperbolicity of the geodesic flow in $\H^n$ holds when considering the geodesic flow in the unit tangent bundle $T^1\H^n$. In contrast, the hyperbolicity of the geodesic flow on $T\H^n$ is not uniform, since it degenerates when $|\dot{\gamma}|_g$ becomes zero, for a geodesic $\gamma$. Let $\alpha>0$. Our first result shows exponential decay of the spatial density $\rho(f)$ and its derivatives, for Vlasov fields on $(\H^2,g_{\H^2})$ that have initial data compactly supported on $\mathcal{D}_{\alpha}:=\{g_x(v,v)\geq \alpha^2\}$. Similar pointwise decay estimates in time hold on hyperbolic space in higher dimensions. We describe the modifications required to treat this case in Subsection \ref{subsection_decay_spatial_density_higher_dimension}.

Later, we consider Vlasov fields on a non-trapping asymptotically hyperbolic manifold $(\M,g)$ in dimension two. In this context, we show exponential decay of the spatial density $\rho(f)$ and its derivatives, for Vlasov fields on $(\M,g)$ that have initial data compactly supported on $\mathcal{D}_{\alpha}:=\{g_x(v,v)\geq \alpha^2\}$. Similar pointwise decay estimates in time hold in higher dimensions. We obtain decay rates determined by the Lyapunov exponents of the geodesic flow with respect to the Liouville measure. In other words, we obtain decay rates determined by the rates of expansion and contraction of the differential of the geodesic flow on the stable and unstable subspaces. 

We also show decay estimates for the spatial density induced by Vlasov fields without the support constraint near the zero velocity set $\{g_x(v,v)=0\}$. We address this issue for Vlasov fields on hyperbolic space $(\H^2,g_{\H^2})$. The hyperbolicity of the geodesic flow on $T\H^2$ is not uniform, since it degenerates when $|\dot{\gamma}|_g$ becomes zero. Thus, we do not expect to obtain exponential decay for general Vlasov fields on hyperbolic space. In this setup, we obtain inverse polynomial decay of the spatial density $\rho(f)$ and its derivatives, for Vlasov fields on $\H^2$. Similar pointwise decay estimates in time hold in higher dimensions.

We investigate the Vlasov equation on non-trapping asymptotically hyperbolic manifolds to offer new insights on the study of stability results for geometric collisionless systems, where the corresponding Hamiltonian flows are \emph{hyperbolic}. The decay estimates derived in this article are particularly suitable to address the non-linear stability problem of the vacuum solution for the Vlasov--Poisson system on hyperbolic space. See \cite{VV23, BVV23} for related works on small data solutions for the Vlasov--Poisson system with the unstable trapping potential $\frac{-|x|^2}{2}$, where the corresponding Hamiltonian flows are hyperbolic. A broad class of geometric collisionless systems arise in general relativity, where the corresponding distribution functions satisfy the \emph{relativistic Vlasov equation} (see \cite{A11} for more information). Consider, for instance, the timelike geodesic flow in de Sitter spacetime, which defines a hyperbolic flow. We hope the methods used in this paper will be helpful to understand more complicated collisionless systems.

\subsection{Vlasov fields on non-trapping asymptotically hyperbolic manifolds}

In this section, we put in precise mathematical terms the objects we study through the paper: Vlasov fields on non-trapping asymptotically hyperbolic manifolds.

\subsubsection{The geometric backgrounds}

First, we introduce the Riemannian manifolds $(\M,g)$ where the Vlasov fields studied in this paper are set. \\

\emph{Hyperbolic space} $(\H^n,g_{\H^n})$. We consider hyperbolic space $(\H^n, g_{\H^n})$ written in the model of the upper branch of the hyperboloid $$\H^n=\Big\{ (\cosh r, (\sinh r) \gamma)\in \R_+\times\R^{n}:~\gamma\in \S^{n-1},~r\geq 0 \Big\}$$ contained in Minkowski spacetime $(\R^{n+1},~\eta:=-\mathrm{d}t\otimes \mathrm{d}t+\mathrm{d}x^1\otimes \mathrm{d}x^1+\dots+\mathrm{d}x^n\otimes \mathrm{d}x^n)$, with the Riemannian metric $g_{\H^n}$ induced by the Lorentzian metric $\eta$ of Minkowski spacetime, given by $$g_{\H^n}=\mathrm{d}r\otimes \mathrm{d}r+\sinh^2 r \mathrm{d}\gamma_{\S^{n-1}},$$ in terms of the standard Riemannian metric on the unit sphere $\mathrm{d}\gamma_{\S^{n-1}}$. \\

\emph{Non-trapping asymptotically hyperbolic manifolds $(\M,g)$.} Let $(\M,g)$ be an oriented complete Riemannian manifold with bounded and strictly negative Gaussian curvature $K_g$. In the following, we denote the closed ball of radius $R_0$ centered at the origin of $\H^n$ by $\overline{B(R_0)}$. We consider the following definitions.

\begin{definition}\label{ahyp}
A Riemannian manifold $(\M,g)$ is \emph{asymptotically hyperbolic} if there exists a compact set $K\subset \M$, $R_0>0$, $\beta>2$, and a diffeomorphism $\Psi:\H^n\setminus \overline{B(R_0)}\to \M\setminus K$ such that $$|D_{g_{\H^n}}^j(\Psi^*g-g_{\H^n})|_{g_{\H^n}}=O(e^{-\beta r}),$$
for every $j\in\{0,1,2\}$.  
\end{definition}

\begin{definition}\label{nontrapping}
A Riemannian manifold $(\M,g)$ is \emph{non-trapping} if the orbit under the geodesic flow of any pair $(x,v)\in T^1\M$ is unbounded. 
\end{definition}

In the rest of the paper, when considering Riemannian manifolds with variable curvature, we will focus on non-trapping asymptotically hyperbolic Riemannian manifolds according to the previous definitions.

\subsubsection{The Vlasov equation on a Riemannian manifold}

In this article, we consider a distribution function $f:[0,\infty)\times T\M\to [0,\infty)$ satisfying the Vlasov equation on a non-trapping asymptotically hyperbolic manifold $(\M, g)$. The \emph{Vlasov equation on a Riemannian manifold} $(\M, g)$, with respect to the coordinate system $(x,v)\in T\M$, takes the form
\begin{equation}\label{Vlasov_eqn_coordinates}
\partial_t f+v^{i}\partial_{x^{i}}f-v^{i}v^{j}\Gamma^{k}_{ij}\partial_{v^{k}}f=0,
\end{equation}
where $\Gamma_{ij}^k$ are the Christoffel symbols of $(\M, g)$. We consider initial data for the Vlasov equation $f_0:T\M\to [0,\infty)$. The Vlasov equation on a complete $C^k$ Riemannian manifold $(\M,g)$ is \emph{well-posed} with initial data $f_0\in C_{x,v}^j(T\M)$ for every $j\in \{0,1,\dots , k\}$. The well-posedness of the Vlasov equation on a complete $C^k$ Riemannian manifold follows directly by the representation formula of Vlasov fields in terms of the geodesic flow, by using the regularity of the flow map. 

In the specific case of hyperbolic space $(\H^2,g_{\H^2})$, the Vlasov equation takes the form $$\partial_t f+v^{r}\partial_{r}f+v^{\theta}\partial_{\theta}f+(v^{\theta})^2\sinh r\cosh r \partial_{v^{r}}f-2v^rv^{\theta}\coth r\partial_{v^{\theta}}f=0,$$ in terms of the Christoffel symbols of $\H^2$ in the coordinate system previously introduced. See Subsection \ref{subsection_sasaki_hyperbolic_space} for the precise Christoffel symbols of $(\H^2, g_{\H^2})$ in the hyperboloidal coordinate system. The Vlasov equation on $(\H^2,g_{\H^2})$ is well-posed with initial data $f_0\in C_{x,v}^j(T\M)$ for every $j\in \N_0$.

\subsection{The main results}

In this subsection, we present the main decay estimates for the Vlasov fields studied in this paper. First, we consider Vlasov fields on hyperbolic space, and later Vlasov fields on non-trapping asymptotically hyperbolic manifolds.

From now on, we use the notation $A\lesssim B$ to specify that there exists a universal constant $C > 0$ such that $A \leq CB$, where $C$ depends only on the dimension $n$, the corresponding order of regularity, or other fixed constants.

\subsubsection{Decay for Vlasov fields on hyperbolic space}

In this section, we state pointwise decay estimates in time for the \emph{spatial density} induced by a solution $f(t,x,v)$ to the Vlasov equation on hyperbolic space $(\H^2,g_{\H^2})$, given by $$\rho(f)(t,x)=\int f(t,x,v)\sinh r \mathrm{d}v^r\mathrm{d}v^{\theta},$$ in terms of the volume form on $(\H^2,g_{\H^2})$ written in local coordinates. In order to avoid the region of phase space $\{g_x(v,v)=0\}$ where dispersion degenerates, we define an invariant subset $\D_{\alpha}$ of the tangent bundle, given by $$\D_{\alpha}:=\Big\{ (x,v)\in T\H^2:~g_x(v,v)\geq \alpha^2 \Big\},$$ where we will assume the initial distribution function is supported. In the following, we take derivatives of the spatial density using the normalized frame $\{\partial_{r},\sinh^{-1}r\partial_{\theta}\}$ on $(\H^2,g_{\H^2})$.

\begin{theorem}[Exponential decay for Vlasov fields supported on $\D_{\alpha}$] \label{thm_decay_spatial_derivatives_hyperbolic}
Let $\alpha>0$. Let $f_0\in C_{x,v}^1(T\H^2)$ be an initial data for the Vlasov equation on hyperbolic space that is compactly supported on $\D_{\alpha}$. Then, the spatial density induced by the corresponding Vlasov field $f$ satisfies
\begin{align}
        &\qquad |\rho(f)(t,x)|\lesssim \dfrac{1}{\exp(\alpha t)}\|f_0\|_{L^{\infty}_{x,v}},\label{estimate_spatial_density_no_derivatives}\\
        \Big|\dfrac{1}{\sinh r}\partial_{\theta}\rho(f)(t,x)&\Big|\lesssim \dfrac{1}{\exp(2\alpha t)}\|f_0\|_{W^{1,\infty}_{x,v}},\qquad |\partial_{r}\rho(f)(t,x)|\lesssim \dfrac{1}{\exp(\alpha t)}\|f_0\|_{W^{1,\infty}_{x,v}},\label{estimate_spatial_density_with_derivatives}
\end{align}
for every $t\geq 0$ and every $x\in \H^2$.
\end{theorem}

\begin{remark}
\begin{enumerate}[(a)]
    \item We obtain \emph{optimal decay rates} for the spatial density and its first order derivatives. The decay rates are expressed in terms of the \emph{minimal Lyapunov exponent} $\alpha$ associated to the geodesics on the support of the distribution function. We also emphasize the difference between the decay rates obtained for the radial and angular derivatives of the spatial density. This discrepancy comes from weights that arise when estimating the derivatives of the spatial density in terms of the commuting vector fields for the Vlasov equation. See Subsection \ref{subsection_decay_derivatives_spatial_density_dimension_two} for further details.
    \item The decay rates in Theorem \ref{thm_decay_spatial_derivatives_hyperbolic} degenerate when $\alpha$ becomes zero. This holds because the hyperbolic expansion/contraction associated to the geodesics in hyperbolic space degenerates when $g_x(v,v)$ becomes zero. For this reason, Vlasov fields on hyperbolic space do not decay exponentially in time for general Vlasov fields. This behavior is compatible with inverse polynomial decay in time.  
\end{enumerate}
\end{remark}

Similar decay estimates in time hold for the spatial density on hyperbolic space in higher dimensions. We describe the modifications required to treat this case in Subsection \ref{subsection_decay_spatial_density_higher_dimension}. 

\begin{theorem}[Polynomial decay for Vlasov fields on hyperbolic space]\label{thm_decay_spatial_derivatives_hyperbolic_no_suppt_restrict}
Let $f_0\in C_{x,v}^1(T\H^2)$ be a compactly supported initial data for the Vlasov equation on hyperbolic space. Then, the spatial density induced by the corresponding Vlasov field $f$ satisfies
\begin{align}
        &\qquad |\rho(f)(t,x)|\lesssim \dfrac{1}{t^2} \|f_0\|_{L^{\infty}_{x,v}},\label{estimate_spatial_density_no_derivatives_no_suppt_restrict}\\
        \Big|\dfrac{1}{\sinh r}\partial_{\theta}\rho(f)(t,x)&\Big|\lesssim \dfrac{1}{t}\|f_0\|_{W^{1,\infty}_{x,v}},\qquad |\partial_{r}\rho(f)(t,x)|\lesssim \dfrac{1}{t}\|f_0\|_{W^{1,\infty}_{x,v}},\label{estimate_spatial_density_with_derivatives_no_suppt_restrict}
\end{align}
for every $t\geq 1$ and every $x\in \H^2$.
\end{theorem}

\begin{remark}
We obtain \emph{optimal decay rates} for the spatial density and its first order derivatives. The decay rate of $\rho(f)$ \emph{coincides} with the corresponding rate for Vlasov fields on Euclidean space.  In contrast, the decay rate of derivatives of $\rho(f)$ is \emph{slower} than the corresponding rate for Vlasov fields on Euclidean space. The precise decay rates come from particle energy weights that appear when using the commuting vector fields for the Vlasov equation. See Subsection \ref{subsection_decay_compact_support_vlasov_hyperbolic} for more details.
\end{remark}

\subsubsection{Decay for Vlasov fields on non-trapping asymptotically hyperbolic manifolds}

Let $\kappa_1>\kappa_2>0$. Let $(\M,g)$ be asymptotically hyperbolic, non-trapping, and with Gaussian curvature $-\kappa_1 \leq K_g\leq -\kappa_2$. 

In this section, we state pointwise decay estimates in time for the \emph{spatial density} induced by a solution $f(t,x,v)$ to the Vlasov equation on $(\M,g)$, given by $$\rho(f)(t,x)=\int f(t,x,v)\sqrt{\det g}\mathrm{d}v^r\mathrm{d}v^{\theta},$$ in terms of the explicit volume form on $(\M,g)$. In order to avoid the region of phase space $\{g_x(v,v)=0\}$ where dispersion degenerates, we define an invariant subset $\D_{\alpha}$ of the tangent bundle, given by $$\D_{\alpha}:=\Big\{ (x,v)\in T\M:~g_x(v,v)\geq \alpha^2 \Big\},$$ where we will assume the initial distribution function is supported. In the following, we take derivatives of the spatial density using the normalized frame $\{\partial_{r},\sinh^{-1}r\partial_{\theta}\}$ on $(\M,g)$.

\begin{theorem}[Exponential decay for Vlasov fields supported on $\D_{\alpha}$] \label{thm_decay_spatial_density_ah_manif}
Let $\alpha>0$. Let $f_0\in C_{x,v}^1(T\M)$ be an initial data for the Vlasov equation on $(\M^2,g)$ that is compactly supported on $\D_{\alpha}$. Then, the spatial density induced by the corresponding Vlasov field $f$ satisfies
\begin{align}
        &\qquad |\rho(f)(t,x)|\lesssim \dfrac{1}{\exp(\alpha t)}\|f_0\|_{L^{\infty}_{x,v}},\label{estimate_spatial_density_no_derivatives_ah}\\
        \Big|\dfrac{1}{\sinh r}\partial_{\theta}\rho(f)(t,x)&\Big|\lesssim \dfrac{1}{\exp(2\alpha t)}\|f_0\|_{W^{1,\infty}_{x,v}},\qquad |\partial_{r}\rho(f)(t,x)|\lesssim \dfrac{1}{\exp(\alpha  t)}\|f_0\|_{W^{1,\infty}_{x,v}},\label{estimate_spatial_density_with_derivatives_ah}
\end{align}
for every $t\geq 0$ and every $x\in \M$.
\end{theorem}
\begin{remark}
\begin{enumerate}
\item We obtain optimal decay rates for the spatial density and its first order derivatives. The decay rates for the spatial density induced by Vlasov fields on $(\M,g)$ coincides with the decay rates for Vlasov fields on hyperbolic space. We make key use of the rate of convergence of the metric $g$ to the one in hyperbolic space $g_{\H^2}$ at infinity. We do not require the metric to be a perturbation of hyperbolic space. The proof of Theorem \ref{thm_decay_spatial_density_ah_manif} exploits a commuting vector field in the unstable invariant distribution of phase space. 
\item The decay rates in Theorem \ref{thm_decay_spatial_density_ah_manif} degenerate when $\alpha$ becomes zero. This holds because the hyperbolic expansion/contraction associated to the geodesics in $(\M,g)$ degenerates when $g_x(v,v)$ becomes zero. For this reason, Vlasov fields on a non-trapping asymptotically hyperbolic Riemannian manifold $(\M,g)$ do not decay exponentially in time for general Vlasov fields.  
\end{enumerate}
\end{remark}

We finish the paper with a general (non-optimal) decay estimate for the spatial density induced by a Vlasov field on an asymptotically hyperbolic manifold. See Appendix \ref{app_subsection_decay_spatial_density_AH} for more details.

\subsubsection{Previous stability results for collisionless systems on Riemannian manifolds}

As commented earlier, Strichartz estimates were obtained for Vlasov fields on Euclidean space by Castella and Perthame \cite{CasPer}. After this work, Salort \cite{S07} studied Vlasov fields on non-trapping asymptotically flat manifolds using approximating arguments. \cite{S07} also derived Strichartz estimates for Vlasov fields on compact Riemannian manifolds with methods introduced by Bahouri--Chemin \cite{BC99} and Burq--G\'erard--Tzvetkov \cite{BGT} for the study of wave equations. See also \cite{S06}.

To the knowledge of the authors, the only stability result for Vlasov fields on hyperbolic space is given in the work of Diacu, Ibrahim, Lind, and Shen \cite{DILS16}. In that paper, the authors extend the classical Vlasov--Poisson system on Euclidean space $\R^n$ to hyperbolic space $\H^n$, and the unit sphere $\S^n$. \cite{DILS16} studies the local well-posedness properties of the Vlasov--Poisson system on these geometric backgrounds. This paper also derives the Vlasov--Poisson system in the corresponding two dimensional cases by using works on the $N$-body problem by Diacu \cite{D12, D14}. Moreover, the authors derive Penrose-type stability conditions in order to obtain the linear stability of homogeneous stationary solutions of the Vlasov--Poisson system on these specific Riemannian manifolds. 

\subsection{Ingredients of the proof}

In this paper, we prove decay estimates in time for the spatial density on non-trapping asymptotically hyperbolic manifolds  by proving decay in time for the size of the velocity support of the distribution function. The decay of the spatial density on holds due to the lack of recurrence of the geodesic flow in these geometric backgrounds. We suitably integrate the geodesic flow to estimate the velocity coordinates in time. The proof of the decay of $\rho(f)$ follows from these arguments. We only require for the initial distribution function to have enough regularity for the induced spatial density to be well-defined, in order to show the corresponding decay estimates. 

Moreover, we prove decay estimates in time for derivatives of the spatial density on non-trapping asymptotically hyperbolic manifolds by using commuting vector fields for the Vlasov equation. As a result, we reduce the estimates for derivatives of $\rho(f)$ to the decay in time of the velocity support of the distribution. First, we write a derivative of the spatial density, as the spatial density of a derivative of the distribution function by 
\begin{equation}\label{identity_derivative_spatial_density}
        \partial_{x^i}\rho(f)(t,x)=\int (\partial_{x^i}-v^j\Gamma_{ij}^k\partial_{v^k})f\dvol_{T_x\M}(v)=:\rho\Big(\Hor_{(x,v)}(\partial_{x^i})f\Big)(t,x),
\end{equation}
in terms of the \emph{horizontal lift} $\Hor_{(x,v)}(\partial_{x^i})$ of the corresponding vector field $\partial_{x^i}$. See Section \ref{section_preliminaries_hyperbolic} for more information about horizontal lifts. Later, we use a suitable class $\lambda$ of commuting vector fields for the Vlasov equation, that arises by studying the Jacobi equation on the tangent bundle of the corresponding Riemannian manifold, with respect to the Sasaki metric. In particular, we exploit the commuting vector fields for the Vlasov equation that are contained in the unstable invariant distribution of phase space. We write the horizontal lifts in terms of commuting vector fields plus errors than can be controlled after integration by parts in the fibers of the tangent bundle. The proof of the decay estimates \eqref{estimate_spatial_density_with_derivatives} are finally obtained using the decay of the size of the velocity support of the distribution function.

\subsection{Outline of the paper}
The remainder of the paper is structured as follows.
\begin{itemize}
    \item \textbf{Section \ref{section_preliminaries_hyperbolic}}. We review the Jacobi fields along the geodesic flow in hyperbolic space, and also in negatively curved manifolds. We set commuting vector fields for the Vlasov equation on hyperbolic space, and also on negatively curved manifolds.
    \item \textbf{Section \ref{section_dispersive_hyperbolic}.}
    We address decay estimates for Vlasov fields on hyperbolic case. The proofs of Theorem \ref{thm_decay_spatial_derivatives_hyperbolic} and Theorem \ref{thm_decay_spatial_derivatives_hyperbolic_no_suppt_restrict} are obtained. 
    \item \textbf{Section \ref{section_dispersive_perturbations}.} We address decay estimates for Vlasov fields on non-trapping asymptotically hyperbolic case. The proof of Theorem \ref{thm_decay_spatial_density_ah_manif} is obtained.
\item \textbf{Appendix \ref{app_subsection_decay_spatial_density_AH}.} A proof of (non-optimal) decay for Vlasov fields on asymptotically hyperbolic manifolds supported on $\D_{\alpha}$ is provided.
\end{itemize}

\begin{ack}
RVR would like to express his gratitude to his supervisors Mihalis Dafermos and Cl\'ement Mouhot for their continued guidance and encouragements. RVR would like to thank Jacques Smulevici for several helpful discussions. RVR received funding from the ANID grant 72190188, the Cambridge Trust grant 10469706, and the European Union’s Horizon 2020 research and innovation programme under the Marie Skłodowska-Curie grant 101034255. AVR received funding from the grant FONDECYT Iniciaci\'on 11220409.
\end{ack}

\section{Preliminaries: Jacobi fields and commuting vector fields}\label{section_preliminaries_hyperbolic}

In this section, we introduce the setup and tools required in the proofs of our main results. Firstly, we introduce the Sasaki metric on the tangent bundle of a Riemannian manifold. Secondly, we review the hyperbolicity of the geodesic flow in hyperbolic space, and negatively curved surfaces. We finish this section setting commuting vector fields for the study of the Vlasov equation.

\subsection{The Sasaki metric on the tangent bundle}\label{subsection_sasaki_unit_tangent_bundle_hyperbolic}

Let $(\M,g)$ be a Riemannian manifold. The \emph{Sasaki metric} $\overline{g}$ is a Riemannian metric induced on $T\M$. To define this Riemannian structure, we recall the decomposition of the tangent space of the tangent bundle at a point $(x,v)\in T\M$ given by $$T_{(x,v)}T\M=\mathcal{H}_{(x,v)}\oplus \mathcal{V}_{(x,v)}$$ into the \emph{horizontal subspace} $\mathcal{H}_{(x,v)}$ and the \emph{vertical subspace} $\mathcal{V}_{(x,v)}$ defined as
\begin{align*}
    \mathcal{H}_{(x,v)}&:=\Hor_{(x,v)}(T_x\M)=\{\Hor_{(x,v)}(Y):~Y\in T_x\M\},\\
    \mathcal{V}_{(x,v)}&:=\Ver_{(x,v)}(T_x\M)=\{\Ver_{(x,v)}(Y):~Y\in T_x\M\},
\end{align*}
where the \emph{horizontal lift} $\Hor_{(x,v)}:T_x\M \to T_{(x,v)}T\M$ and the \emph{vertical lift} $\Ver_{(x,v)}:T_x\M \to T_{(x,v)}T\M$ are defined in local coordinates as $$\Hor_{(x,v)}(Y^{i}\partial_{x^{i}}):=Y^{i}\partial_{x^{i}}-Y^{i}v^{j}\Gamma_{ij}^{k}\partial_{v^k},\qquad
\Ver_{(x,v)}(Y^{i}\partial_{x^{i}}):=Y^{i}\partial_{v^{i}},$$ for an arbitrary vector field $Y^{i}\partial_{x^{i}}\in T\M$.

The vertical subspace $\mathcal{V}_{(x,v)}$ can be defined as the kernel of the differential $d\pi:TT\M\to T\M$ of the canonical projection $\pi:T\M\to\M$. Moreover, the horizontal subspace $\mathcal{H}_{(x,v)}$ can be defined as the kernel of the \emph{connection map} $K:TT\M\to T\M$ defined in terms of the Levi-Civita connection. See \cite[Chapter 1]{Pat99} for further details about the connection map. 

\begin{definition}\label{definition_sasaki_metric_general_riemannian}
Let $(\M,g)$ be a Riemannian manifold. We define the \emph{Sasaki metric} $\bar{g}$ on the tangent bundle $T\M$ as the unique metric for which
\begin{align*}
    \bar{g}_{(x,v)}(\Hor_{(x,v)}(Y),\Hor_{(x,v)}(Z))&=g_x(Y,Z),\\
    \bar{g}_{(x,v)}(\Hor_{(x,v)}(Y),\Ver_{(x,v)}(Z))&=0,\\
    \bar{g}_{(x,v)}(\Ver_{(x,v)}(Y),\Ver_{(x,v)}(Z))&=g_x(Y,Z),
\end{align*}
for every $(x,v)\in T\M$, and every $Y,Z\in T_x\M$.
\end{definition}

\begin{remark}
The \emph{generator of the geodesic flow} on a Riemannian manifold $(\M,g)$ is given by the horizontal vector field $$\Hor_{(x,v)}(v)=v^{i}\partial_{x^{i}}-v^{i}v^{j}\Gamma^{k}_{ij}\partial_{v^{k}}.$$ The Vlasov equation on a Riemannian manifold $(\M,g)$ is written in terms of this horizontal vector field. 
\end{remark}

\subsubsection{The Sasaki metric on the tangent bundle of hyperbolic space}\label{subsection_sasaki_hyperbolic_space}

In hyperbolic space $(\H^2,g_{\H^2})$, written in the model of the upper branch of a hyperboloid, the non-trivial Christoffel symbols are given by $$ \Gamma_{\theta\theta}^r=-\cosh r\sinh r,\qquad \Gamma_{r\theta}^{\theta}=\coth r.$$ 

Lifting the vector fields $\partial_r$ and $\partial_{\theta}$ into the tangent of the tangent bundle of hyperbolic space, we obtain
\begin{align}
\begin{aligned}\label{definition_coordinates_horizontal_vertical_hyperbolic_space}
    \Hor_{(x,v)}(\partial_{r})&=\partial_{r}-\coth rv^{\theta}\partial_{v^{\theta}},\\
    \Hor_{(x,v)}(\partial_{\theta})&=\partial_{\theta}+\cosh r\sinh rv^{\theta}\partial_{v^r}-\coth rv^r\partial_{v^{\theta}},\\
    \Ver_{(x,v)}(\partial_{r})&=\partial_{v^r},\\
    \Ver_{(x,v)}(\partial_{\theta})&=\partial_{v^{\theta}},
\end{aligned}
\end{align}
by using the Christoffel symbols written above. Then, the Sasaki metric $\bar{g}_{\H^2}$ on the tangent bundle of hyperbolic space $T\H^2$ can be set.

In the following, we use the orthonormal frame on the tangent bundle of hyperbolic space given by $$\Hor_{(x,v)}(\partial_{r}),\quad \Hor_{(x,v)}\Big(\dfrac{\partial_{\theta}}{\sinh r}\Big), \quad \Ver_{(x,v)}(\partial_{r}),\quad \Ver_{(x,v)}\Big(\dfrac{\partial_{\theta}}{\sinh r}\Big).$$ The Vlasov equation on hyperbolic space $(\H^2,g_{\H^2})$ can be written using this frame as $$ \partial_t f+v^r\Hor_{(x,v)}(\partial_{r})f+v^{\theta}\sinh r\Hor_{(x,v)}\Big(\dfrac{\partial_{\theta}}{\sinh r}\Big)f=0,$$ in terms of the previous orthonormal frame on the tangent bundle of hyperbolic space.

\subsection{Jacobi fields along the geodesic flow}

Given a Riemannian manifold $(\M,g)$, we recall the \emph{geodesic flow map} $\phi_t: T\M\to T\M$ defined as the mapping $t\mapsto \phi_t(x,v)$ that determines the unique geodesic with initial data $(x,v)$. In this subsection, we first recall the definition of Jacobi fields and relate them to the differential of the geodesic flow map. Later, we will describe the behavior of Jacobi fields on hyperbolic space, and negatively curved surfaces.

\begin{definition}
Let $\e>0$. Let $\gamma_{\tau}:I\to \M$ be a one-parameter family of geodesics in $\M$, where $\tau\in (-\e,\e)$ and $\gamma:=\gamma_0$. A vector field $J(t)\in T_{\gamma(t)}\M$ of the form $$J(t)=\dfrac{\partial \gamma_{\tau}}{\partial\tau}(t)\Big|_{\tau=0}$$ is called a \emph{Jacobi field in $(\M,g)$}. A Jacobi field $J$ satisfies the \emph{Jacobi equation} $$\nabla_{\dot{\gamma}}\nabla_{\dot{\gamma}} J=R(\dot{\gamma},J)\dot{\gamma}.$$
\end{definition}
 
The Jacobi fields in $(\M,g)$ are \emph{generated} by the differential of the geodesic flow. Let us make this statement precise. We recall the canonical projection $\pi:T\M\to\M$ of the tangent bundle $T\M$ into the manifold $\M$. We define an \emph{adapted curve $c_V$ to a vector} $V\in T_{(x,v)}T\M$ as an arbitrary curve $c_V:(-\epsilon,\epsilon)\to T\M$ such that $$c_V(0)=(x,v),\qquad \dot{c}_V(0)=V.$$ The Jacobi fields in $(\M,g)$ are \emph{generated} by the differential of the geodesic flow, in the sense that the map $(t,\tau)\mapsto \pi(\phi_t(c_V(\tau)))$ defines a variation of geodesics that induces the Jacobi field $$J_V(t)=\dfrac{\partial}{\partial \tau}\pi (\phi_t (c_V(\tau)))\Big|_{\tau=0},$$ whose initial conditions are  $J_V(0)=\mathrm{d}\pi(x,v)(V)$ and $\nabla_{\dot{\gamma}}J_V(0)=K(x,v)(V)$.
 
We also consider \emph{Jacobi fields in the tangent bundle} $(T\M,\bar{g})$. For our purposes, we relate the Jacobi fields in the tangent bundle with the differential of the geodesic flow map. 

\begin{definition}
Let $\e>0$. Let $\bar{\gamma}_{\tau}:I\to T\M$ be a one-parameter family of geodesics in $T\M$, where $\tau\in (-\e,\e)$ and $\bar{\gamma}:=\bar{\gamma}_0$. A vector field $\bar{J}(t)\in T_{\bar{\gamma}(t)}T\M$ of the form $$\bar{J}(t):=\dfrac{\partial \bar{\gamma}_{\tau}}{\partial\tau}(t)\Big|_{\tau=0},$$ is called a \emph{Jacobi field in $(T\M,\bar{g})$}. A Jacobi field $\bar{J}$ satisfies the \emph{Jacobi equation} $$\overline{\nabla}_X\overline{\nabla}_X \bar{J}=\bar{R}(X,\bar{J})X,$$ where $\overline{\nabla}$ is the Levi-Civita connection on $(T\M,\overline{g})$, $\bar{R}$ is the Riemann curvature tensor on $(T\M,\overline{g})$, and $X$ is the generator of the geodesic flow.
\end{definition}

Similarly to the case of Jacobi fields in $(\M,g)$, the Jacobi fields in $(T\M,\bar{g})$ are \emph{generated} by the differential of the geodesic flow. Consider an adapted curve $c_V$ to a vector $V\in T_{(x,v)}T\M$. The map $(t,\tau)\mapsto \phi_t(c_V(\tau))$ defines a variation of geodesics in $(T\M,\bar{g})$ that induces a Jacobi field $$\bar{J}_V(t)=\dfrac{\partial}{\partial \tau}\phi_t(c_V(\tau))\Big|_{\tau=0}=\mathrm{d}\phi{_t}(x,v)(V).$$

The following lemma establishes the precise relation between the Jacobi fields in a Riemannian manifold $(\M,g)$ and the Jacobi fields in the tangent bundle $(T\M,\bar{g})$. 

\begin{lemma}\label{lemma_action_differential_flow_hyperbolic}
The differential of the geodesic flow map $\phi_t: T\M\to T\M$ satisfies $$\mathrm{d}\phi{_t}(x,v)(V)=\mathrm{Hor}_{\phi_t(x,v)}(J_V(t))+\mathrm{Ver}_{\phi_t(x,v)}(\nabla_{\dot{\gamma}}J_V(t)),  $$ for every $t\in \R$, every $(x,v)\in T\M$, and every vector $V\in T_{(x,v)}T\M$.
\end{lemma}

See \cite[Lemma 1.40]{Pat99} for a proof of Lemma \ref{lemma_action_differential_flow_hyperbolic}.

\subsubsection{Jacobi fields in hyperbolic space}

In hyperbolic space $(\H^2,g_{\H^2})$, written in the model of the upper branch of a hyperboloid, we consider an orthogonal parallel frame along an arbitrary geodesic $\gamma$ given by 
\begin{equation}\label{frame_parallel_orthogonal_dimension_two_hyperbolic_space}
\dot{\gamma}:=v^r\partial_{r}+\sinh r v^{\theta}\dfrac{\partial_{\theta}}{\sinh r}, \qquad N:=-\sinh r v^{\theta}\partial_{r}+v^{r}\dfrac{\partial_{\theta}}{\sinh r}.    
\end{equation}
We write the Jacobi equation on hyperbolic space as a linear system of odes for the components of $J$ in terms of the moving frame \eqref{frame_parallel_orthogonal_dimension_two_hyperbolic_space}. We use that the Riemann curvature tensor along a geodesic $\gamma$ satisfies $$R(\dot{\gamma},N)\dot{\gamma}=|\dot{\gamma}|_g^2N,\qquad R(\dot{\gamma}, \dot{\gamma})\dot{\gamma}=0.$$

\begin{proposition}\label{proposition_jacobi_eqn_hyperbolic_space}
Let $J$ be a Jacobi field in hyperbolic space $(\H^2,g_{\H^2})$. Then, the Jacobi equation satisfied by the components of the Jacobi field $J=J^0\dot{\gamma}+J^N N$ is reduced to
\begin{equation}\label{Jacobi_eqn_hyperbolic_space_precise}
    \dfrac{\mathrm{d}^2J^0}{\mathrm{d}t^2}=0,\qquad \dfrac{\mathrm{d}^2J^N}{\mathrm{d}t^2}=|\dot{\gamma}|_g^2J^N.
\end{equation}
\end{proposition}

In the tangent bundle of hyperbolic space $(T\H^2,\bar{g}_{\H^2})$, we consider a moving frame along an arbitrary geodesic $\bar{\gamma}$, given by $$\Hor_{(x,v)}(v),\qquad \Hor_{(x,v)}(N),\qquad \Ver_{(x,v)}(v), \qquad
    \Ver_{(x,v)}(N).$$
By combining the Jacobi equation on hyperbolic space \eqref{Jacobi_eqn_hyperbolic_space_precise} with Lemma \ref{lemma_action_differential_flow_hyperbolic}, we obtain the Jacobi fields in $(T\H^2,\bar{g}_{\H^2})$. In the following, we set $$J_u^N:=\frac{1}{2}\Big(J^N+\frac{1}{|\dot{\gamma}|_g}\frac{\mathrm{d}J^N}{\mathrm{d}t}\Big),\qquad J_s^N:=\frac{1}{2}\Big(J^N-\frac{1}{|\dot{\gamma}|_g}\frac{\mathrm{d}J^N}{\mathrm{d}t}\Big).$$

\begin{proposition}\label{proposition_jacobi_fields_tangent_bundle}
Let $\bar{J}$ be a Jacobi field in the tangent bundle of hyperbolic space $(T\H^2,\bar{g}_{\H^2})$. Then, the Jacobi equation satisfied by the components of the Jacobi field 
\begin{align*}
\bar{J}(t)&=J^0 \mathrm{Hor}_{(x,v)}(v)+\dfrac{\mathrm{d}J^0}{\mathrm{d}t}\mathrm{Ver}_{(x,v)}(v)\\
&\qquad + J_u^N \Big(\mathrm{Hor}_{(x,v)}(N)+|\dot{\gamma}|_g \mathrm{Ver}_{(x,v)}(N)\Big)+ J_s^N \Big(\mathrm{Hor}_{(x,v)}(N)-|\dot{\gamma}|_g \mathrm{Ver}_{(x,v)}(N)\Big)
\end{align*}
is reduced to 
\begin{equation}\label{Jacobi_eqn_hyperbolic_space_tangent_bundle_precise}
    \dfrac{\mathrm{d}^2J^0}{\mathrm{d}t^2}=0,\qquad \dfrac{\mathrm{d}J_u^N}{\mathrm{d}t^2}=|\dot{\gamma}|_gJ_u^N,\qquad \dfrac{\mathrm{d}J_s^N}{\mathrm{d}t^2}=-|\dot{\gamma}|_gJ_s^N
\end{equation}
\end{proposition}

Proposition \ref{proposition_jacobi_fields_tangent_bundle} shows that the geodesic flow in $(\H^2,g_{\H^2})$ is \emph{hyperbolic}, in the sense that the restriction of the geodesic flow to the unit tangent bundle $\phi_t: (T^1\H^2, \overline{g}|_{T^1\H^2})\to (T^1\H^2, \overline{g}|_{T^1\H^2})$ defines a hyperbolic flow according to \cite[Chapter 17, Section 4]{KH}. In other words, there exists an invariant decomposition of the tangent of the unit tangent bundle $$T_{(x,v)}T^1\H^2=E_0(x,v)\oplus E_u(x,v)\oplus E_s(x,v),$$ into the subspace $E_0(x,v):=\spann\{\mathrm{Hor}_{(x,v)}(v)\}$ spanned by the generator of the geodesic flow, the \emph{stable subspace} $E_s(x,v):=\spann\{\mathrm{Hor}_{(x,v)}(N)-\mathrm{Ver}_{(x,v)}(N)\}$, and the \emph{unstable subspace} $E_u(x,v):=\spann\{\mathrm{Hor}_{(x,v)}(N)+\mathrm{Ver}_{(x,v)}(N)\}$. The differential of the geodesic flow contracts exponentially the distribution $E_s(x,v)$, and expands exponentially the distribution $E_u(x,v)$.

\begin{remark}
The Jacobi equation on the tangent bundle of hyperbolic space $(T\H^n,\bar{g}_{\H^n})$ in higher dimensions can be written similarly to \eqref{Jacobi_eqn_hyperbolic_space_tangent_bundle_precise}. Consider an orthonormal parallel frame $\{\dot{\gamma}, N_1, N_2,\dots, N_{n-1}\}$ along a geodesic $\gamma$. Set $$J_u^i:=\frac{1}{2}\Big(J^i+\frac{1}{|\dot{\gamma}|_g}\frac{\mathrm{d}J^i}{\mathrm{d}t}\Big),\qquad J_s^i:=\frac{1}{2}\Big(J^i-\frac{1}{|\dot{\gamma}|_g}\frac{\mathrm{d}J^i}{\mathrm{d}t}\Big),$$ for every $i\in\{1,2,\dots, n-1\}$. Then, the Jacobi equation satisfied by the components of the Jacobi field
\begin{align*}
\bar{J}(t)&=J^0 \mathrm{Hor}_{(x,v)}(v)+\dfrac{\mathrm{d}J^0}{\mathrm{d}t}\mathrm{Ver}_{(x,v)}(v)\\
&\qquad \sum_{i=1}^{n-1}J_u^i \Big(\mathrm{Hor}_{(x,v)}(N_i)+|\dot{\gamma}|_g\mathrm{Ver}_{(x,v)}(N_i)\Big)+J_s^i \Big(\mathrm{Hor}_{(x,v)}(N_i)-|\dot{\gamma}|_g\mathrm{Ver}_{(x,v)}(N_i)\Big)
\end{align*} 
 is reduced to $$\dfrac{\mathrm{d}^2J^0}{\mathrm{d}t^2}=0,\qquad     \dfrac{\mathrm{d}J_u^i}{\mathrm{d}t}=|\dot{\gamma}|_gJ_u^i,\qquad     \dfrac{\mathrm{d}J_s^i}{\mathrm{d}t}=-|\dot{\gamma}|_gJ_s^i$$ for every $i\in\{1,2,\dots, n-1\}$. In particular, the geodesic flow in hyperbolic space $(\H^n,g_{\H^n})$ is also hyperbolic.
\end{remark}

\subsubsection{Jacobi fields in a negatively curved manifold}\label{subsec_jacobi_neg}

Let $(\M,g)$ be an oriented pinched negatively curved Riemannian manifold. Let $\gamma$ be a geodesic in $(\M,g)$. We consider the unique positively oriented frame $\{\dot{\gamma}, N\}$ along $\gamma$ such that $$|\dot{\gamma}|_g=|N|_g,\qquad g(\dot{\gamma},N)=0.$$ The frame $\{\dot{\gamma}, N\}$ contains two parallel vector fields along $\gamma$. We write the Jacobi equation on $(\M,g)$ as a linear system of odes for the components of $J$ in terms of the moving frame $\{\dot{\gamma}, N\}$. We use that the Riemann curvature tensor along a geodesic $\gamma$ satisfies $$R(\dot{\gamma},N)\dot{\gamma}=-K_g |\dot{\gamma}|_g^2N,\qquad R(\dot{\gamma}, \dot{\gamma})\dot{\gamma}=0,$$ in terms of the Gaussian curvature $K_g$ of $(\M,g)$.

\begin{proposition}\label{proposition_jacobi_eqn_ah}
Let $J$ be a Jacobi field in $(\M,g)$. Then, the Jacobi equation satisfied by the components of the Jacobi field $J=J^0\dot{\gamma}+J^N N$ is reduced to $$\dfrac{\mathrm{d}^2J^0}{\mathrm{d}t^2}=0,\qquad \dfrac{\mathrm{d}^2J^N}{\mathrm{d}t^2}=-|\dot{\gamma}|_g^2 K_g J^N.$$
\end{proposition}

We consider the quotient $q(t):=(J^N)^{-1}\dot{J}^N$ that satisfies the Riccati equation 
\begin{equation}\label{riccati_eqn_gen}
\frac{\mathrm{d}q}{\mathrm{d}t}(t)+q^2(t)+|\dot{\gamma}|_g^2K_g(x(t))=0,
\end{equation} 
as long as $J^N(t)\neq 0$. Let $T\in \R$. Set $J^N_T(t,x,v)$ to be the solution of the Jacobi equation along $\gamma$ with $$J^N_T(0,x,v)=1,\qquad J^N_T(T,x,v)=0.$$ By the pinching assumption the metric $g$ does not have conjugate points so the solution $J^N_T(t,x,v)$ does not vanish for $t\neq T$. Consider the induced solution $q_T(t,x,v):=(J_T^N)^{-1}\dot{J}_T^N$ of the Riccati equation \eqref{riccati_eqn_gen}. By the assumption $J^N_T(T,x,v)=0$ the function $q_T(t,x,v)$ is defined for $t<T$ and $\lim_{t\to T}q_T(t,x,v)=-\infty$. By the work of Hopf \cite{H48}, the limits $$q_s(t,x,v):=\lim_{T\to \infty}q_T(t,x,v),\qquad q_u(t,x,v):=\lim_{T\to \infty}q_{-T}(t,x,v),$$ exist for every $(t,x,v)\in\R_t \times T\M$. By construction, the solutions $q_{u}$ and $q_{s}$ of the Riccati equation are invariant, in other words, $q_{u}(t,x,v)=q_{u}(x,v)$ and $q_{s}(t,x,v)=q_{s}(x,v)$. By construction, the functions $q_{s}$ and $q_{u}$ satisfy the Riccati equation 
\begin{align}\label{Riccati_eqn_invariant_geo}
Xq+q^2+|\dot{\gamma}|_g^2K_g=0.
\end{align}
Moreover, the estimate $q_u>0>q_s$ holds uniformly. For the functions $q_s$ and $q_u$, there are corresponding functions $J^N_s$ and $J^N_u$, which satisfy $q_{s}=(J_s^N)^{-1}\dot{J}_s^N$ and $q_{u}=(J_u^N)^{-1}\dot{J}_u^N$, respectively. For general hyperbolic flows (not necessarily geodesic flows), the functions $q_{u}$ and $q_{s}$ are only Hölder continuous after Hirsch, Pugh, and Shub \cite{HPS77}. In the specific case of negatively curved surfaces, the functions $q_{u}$ and $q_{s}$ belong to $C^{2-}(T\M):=\cap_{\delta>0}C^{2-\delta}(T\M)$ by the work of Hurder and Katok \cite{HK90}. 

In the tangent bundle of $(T\M,\bar{g})$, we consider a moving frame along a geodesic $\bar{\gamma}$ given by $$\Hor_{(x,v)}(v),\qquad \Hor_{(x,v)}(N),\qquad \Ver_{(x,v)}(v), \qquad
    \Ver_{(x,v)}(N).$$ We obtain the Jacobi fields in $(T\M,\bar{g})$ by integrating the equations $q_{s}=(J_s^N)^{-1}\dot{J}_s^N$ and $q_{u}=(J_u^N)^{-1}\dot{J}_u^N$, for the functions $J_s^N$ and $J_u^N$, respectively. 

\begin{proposition}\label{proposition_jacobi_fields_tangent_bundle_ah}
Let $\bar{J}$ be a Jacobi field in $(T\M,\bar{g})$. Then, the Jacobi equation satisfied by the components of the Jacobi field 
\begin{align*}
\bar{J}(t)&=J^0 \mathrm{Hor}_{(x,v)}(v)+\dfrac{\mathrm{d}J^0}{\mathrm{d}t}\mathrm{Ver}_{(x,v)}(v)\\
&\qquad+ J^N_u \Big(\mathrm{Hor}_{(x,v)}(N)+q_u\mathrm{Ver}_{(x,v)}(N)\Big)+ J^N_s\Big(\mathrm{Hor}_{(x,v)}(N)+q_s\mathrm{Ver}_{(x,v)}(N)\Big)
\end{align*}
is reduced to $$\dfrac{\mathrm{d}^2J^0}{\mathrm{d}t^2}=0,\qquad \dfrac{\mathrm{d}J_u^N}{\mathrm{d}t}=q_uJ^N_u,\qquad \dfrac{\mathrm{d}J_s^N}{\mathrm{d}t}=q_sJ^N_s.$$
\end{proposition}

Proposition \ref{proposition_jacobi_fields_tangent_bundle_ah} shows that the geodesic flow in $(\M,g)$ is \emph{hyperbolic}, in the sense that the restriction of the geodesic flow to the unit tangent bundle $\phi_t: (T^1\M, \overline{g}|_{T^1\M})\to (T^1\M, \overline{g}|_{T^1\M})$ defines a hyperbolic flow according to \cite[Chapter 17, Section 4]{KH}. In other words, there exists an invariant decomposition of the tangent of the unit tangent bundle $$T_{(x,v)}T^1\M=E_0(x,v)\oplus E_u(x,v)\oplus E_s(x,v),$$ into the subspace $E_0(x,v):=\spann\{\mathrm{Hor}_{(x,v)}(v)\}$ spanned by the generator of the geodesic flow, the \emph{stable subspace} $E_s(x,v):=\spann\{\mathrm{Hor}_{(x,v)}(N)+q_s\mathrm{Ver}_{(x,v)}(N)\}$, and the \emph{unstable subspace} $E_u(x,v):=\spann\{\mathrm{Hor}_{(x,v)}(N)+q_u\mathrm{Ver}_{(x,v)}(N)\}$. The differential of the geodesic flow contracts exponentially the distribution $E_s(x,v)$, and expands exponentially the distribution $E_u(x,v)$.

\begin{remark}
The Jacobi equation on the tangent bundle of a pinched negatively curved Riemannian manifold in higher dimensions has a similar behavior. A similar analysis of a Riccati equation can be performed in higher dimensions. As a result, one also obtains the hyperbolicity of the geodesic flow on $(T\M,\bar{g}_{\M})$. See \cite[Chapter 17, Section 6]{KH}\footnote{This reference addresses compact negatively curved surfaces, however, the same arguments hold for the class of negatively curved surfaces considered here.} for a proof of the hyperbolicity of $(\M,g)$ using invariant cones techniques. 
\end{remark}

\subsection{Commuting vector fields for the Vlasov equation}\label{subsection_commutators_vlasov_hyperbolic}

In this subsection, we introduce two classes of vector fields, one in $(T\H^2,\bar{g}_{\H^2})$, and another one in $(T\M,\bar{g})$. We will later use these classes of vector fields to obtain decay estimates for Vlasov fields. The vector fields introduced in this section arise from the dynamics of the Jacobi fields in $(T\H^2,\bar{g}_{\H^2})$, and the Jacobi fields in $(T\M,\bar{g})$.

\subsubsection{On hyperbolic space}

Let $\gamma$ be a geodesic in hyperbolic space. We consider the orthogonal parallel frame along $\gamma$ given by \eqref{frame_parallel_orthogonal_dimension_two_hyperbolic_space}. Lifting the vector fields in this frame into $TT\H^2$, we obtain the moving frame 
\begin{align*}
    X&:=\Hor_{(x,v)}(v)=v^{r}\partial_{r}+v^{\theta}\partial_{\theta}+(v^{\theta})^2\sinh r\cosh r \partial_{v^{r}}-2v^rv^{\theta}\coth r\partial_{v^{\theta}} ,\\
    Y&:=\Ver_{(x,v)}(v)=v^r\partial_{v^r}+\sinh r v^{\theta}\dfrac{\partial_{v^{\theta}}}{\sinh r},\\
    H&:=\Hor_{(x,v)}(N)=-\sinh r v^{\theta}\partial_{r}+v^{r}\dfrac{\partial_{\theta}}{\sinh r}+\cosh r v^rv^{\theta}\partial_{v^r}+\Big((v^{\theta})^2-\dfrac{(v^{r})^2}{\sinh^2 r}\Big)\cosh r \partial_{v^{\theta}},\\
    V&:=\Ver_{(x,v)}(N)=-\sinh r v^{\theta}\partial_{v^r}+v^{r}\dfrac{\partial_{v^{\theta}}}{\sinh r}.
\end{align*}
The commutators between the generator of the geodesic flow $X$, and the vector fields $H$, $Y$, $V$ are given by
\begin{equation}\label{commuting_relations_hyperbolic}
    [X,Y]=-X,\qquad [X,H]=-E^2V,\qquad [X,V]=-H,
\end{equation}
in terms of the particle energy $E:=|\dot{\gamma}|_g$. The identities \eqref{commuting_relations_hyperbolic} are part of the standard structure equations of the Lie algebra of vector fields on $T\H^2$. 

The commutators \eqref{commuting_relations_hyperbolic} show the following commuting vector fields with the Vlasov equation on hyperbolic space,
\begin{enumerate}[(a)]
    \item generator of the flow $X$,
    \item uniform motion $tX+Y$,
    \item unstable vector field $U:=e^{E t}(H+EV)$,
    \item stable vector field $S:=e^{-E t}(H-EV)$.
\end{enumerate}
We define the class of vector fields $$\lambda:=\Big\{X, tX+Y, U, S\Big\}.$$
The collection of vector fields $\lambda$ will be used later to obtain decay estimates for the Vlasov equation on hyperbolic space.

\begin{lemma}\label{lemma_commuting_vector_fields_hyperboli_dimension_two}
Let $f$ be a regular Vlasov field on hyperbolic space. Then, $Zf$ is also a solution of this equation for every $Z\in \lambda$.
\end{lemma}
\begin{proof} We use the commuting relations \eqref{commuting_relations_hyperbolic} to show that $[\partial_t+X,Z]=0$, for every $Z\in \lambda$. Since $f$ is a Vlasov field on hyperbolic space, we thus have
$$(\partial_t+X)(Zf)=Z(\partial_t+X)f+[\partial_t+X,Z]f=0,$$ and therefore $Zf$ is a solution as well. 
\end{proof}

\begin{remark}
\begin{enumerate}[(a)]   
\item The class of commuting vector fields $\{t\partial_{x^i}+\partial_{v^i}, \partial_{x^i}\}$ for the Vlasov equation on Euclidean space $(\R_x^n,\delta_{E})$, is composed by Jacobi fields in the tangent bundle of Euclidean space $(\R_x^n\times\R_v^n,\bar{\delta}_{E})$ with respect to the Sasaki metric along an arbitrary geodesic in $\R^n_x$. The class of commuting vector fields $\{t\partial_{x^i}+\partial_{v^i}, \partial_{x^i}\}$ for the Vlasov equation on Euclidean space has played an important role in previous stability results of the vacuum solution for the Vlasov--Poisson system on Euclidean space \cite{Sm, Du22}. 
\item We observe that the stable derivative of Vlasov fields on hyperbolic space $(\H^2,g_{\H^2})$ grows exponentially in time. Using the commuting vector field $S$ of the Vlasov equation contained in the stable distribution of phase space, we obtain $$(H-EV)f(t,x,v)=e^{Et}(H-EV)f_0(t,x_0,v_0),$$
in terms of the corresponding point $(x_0,v_0)$ in the support of the initial distribution function $f_0$. This property of Vlasov fields on hyperbolic space contrasts with Vlasov fields on Euclidean space, for which all derivatives decay in time.
\end{enumerate}
\end{remark}

\subsubsection{On a negatively curved manifold}

Let $\gamma$ be a geodesic in $(\M,g)$. We consider the orthogonal parallel frame along the geodesic $\gamma$ given by $\{\dot{\gamma},N\}$. Lifting the vector fields in this frame into $TT\M$, we obtain the moving frame 
\begin{alignat*}{2}
    X:=\Hor_{(x,v)}(v),\quad Y:=\Ver_{(x,v)}(v),\quad  H:=\Hor_{(x,v)}(N),\quad V:=\Ver_{(x,v)}(N).
\end{alignat*}
The commutators between the generator of the geodesic flow $X$, and the vector fields $H$, $Y$, $V$ are given by
\begin{equation}\label{commuting_relations_neg_curv}
    [X,Y]=-X,\qquad [X,H]=K_gE^2V,\qquad [X,V]=-H,
\end{equation}
in terms of the Gaussian curvature $K_g$ and the particle energy $E:=|\dot{\gamma}|_g$. The identities \eqref{commuting_relations_neg_curv} are part of the standard structure equations of the Lie algebra of vector fields on the tangent bundle of a surface. The relations \eqref{commuting_relations_neg_curv} can be obtained by a direct computation in local coordinates. In this two-dimensional case, one can use isothermal coordinates for this purpose. See \cite[Chapter 3]{PSU23} for further details.

Recall the functions $q_u:C^{2-}(T\M)\to (0,\infty)$ and $q_s:C^{2-}(T\M)\to (-\infty,0)$ satisfied by the Riccati equation \eqref{Riccati_eqn_invariant_geo}. These functions set the stable $E_s$ and unstable $E_u$ invariant distributions of phase space. See Proposition \ref{proposition_jacobi_fields_tangent_bundle_ah}. Using the structure equations \eqref{commuting_relations_neg_curv}, and the Riccati equation \eqref{Riccati_eqn_invariant_geo}, we obtain the following commuting relations 
\begin{equation}\label{commut_gud}
[X,H+q_uV]=-q_u(H+q_uV),\qquad [X,H+q_sV]=-q_s(H+q_sV).
\end{equation} 

The commutators \eqref{commuting_relations_neg_curv} and \eqref{commut_gud} show that the following collection of vector fields commute with the Vlasov equation on $(\M,g)$,
\begin{enumerate}[(a)]
    \item generator of the flow $X$,
    \item uniform motion $tX+Y$,
    \item unstable vector field $U:=e^{\int_0^t q_u d\tau}(H+q_uV)$,
    \item stable vector field $S:=e^{\int_0^t q_s d\tau }(H+q_sV)$.
\end{enumerate}
We define the class of vector fields $$\lambda:=\Big\{X, tX+Y, U, S\Big\}.$$
The collection of vector fields $\lambda$ will be used later to obtain decay estimates for Vlasov equation on non-trapping asymptotically hyperbolic manifolds $(\M,g)$.

\begin{lemma}\label{lemma_commuting_vector_fields_negat_curv_dimension_two}
Let $f$ be a regular Vlasov field on a negatively curved manifold $(\M,g)$. Then, $Zf$ is also a solution of this equation for every $Z\in \lambda$.
\end{lemma}

\section{Decay for Vlasov fields on hyperbolic space}\label{section_dispersive_hyperbolic}

In this section, we prove decay estimates for the spatial density induced by a Vlasov field on hyperbolic space. The geodesic flow in hyperbolic space is determined by the geodesic equations 
\begin{align*}
\frac{\mathrm{d}\theta}{\mathrm{d}t}&=v^{\theta},\qquad \frac{\mathrm{d}v^{\theta}}{\mathrm{d}t} =-2\coth r v^rv^{\theta},\\\frac{\mathrm{d}r}{\mathrm{d}t}&=v^{r}, \qquad  \frac{\mathrm{d}v^r}{\mathrm{d}t} =\cosh r \sinh r (v^{\theta})^2.
\end{align*}
It is well-known that the geodesics in hyperbolic space $\H^2$ are characterised as the intersection between the hyperboloid $\H^n$ with the two dimensional linear subspaces of Minkowski spacetime in dimension $2+1$.

The geodesic flow in hyperbolic space can also be viewed as a \emph{Hamiltonian flow} in the cotangent bundle $(T^*\H^2, \bar{g}_{\H^2})$, where the Hamiltonian $H:T^*\H^2\to[0,\infty)$ is given by $$H(x,v):=\dfrac{1}{2}\Big(v_r^2+\frac{1}{\sinh^2 r} v_{\theta}^2\Big).$$ We define the \emph{particle energy} $E: T\H^2\to [0,\infty) $ of a geodesic $\gamma$ by $$E(x,v):=g_x(v,v)^{\frac{1}{2}}=\sqrt{(v^r)^2+\sinh^2 r (v^{\theta})^2},$$ which is a conserved along the geodesic flow. The Hamiltonian $H$ can be written in terms of the particle energy as $H=\frac{1}{2}E^2$. We also define the \emph{angular velocity} $l: T\H^2\to [0,\infty)$ of a geodesic $\gamma$ by $$l(x,v):=g(\partial_{\theta},\dot{\gamma})=v^{\theta}\sinh^{2} r,$$ which is conserved along the geodesic flow since $\partial_{\theta}$ is a Killing field. The particle energy $E$ and the angular velocity $l$ are related by
\begin{equation}\label{identity_relate_particle_energy_angular_velocity_1}
    E^2=(v^r)^2+\dfrac{l^2}{\sinh^2 r},
\end{equation}
for every geodesic on hyperbolic space. We observe that the geodesic flow in hyperbolic space $(T^*\H^2,\bar{g}_{\H^2}, H)$ is a \emph{completely integrable Hamiltonian flow} in the sense of Liouville, due to the existence of the two independent conserved quantities in involution $H$ and $l$.

\begin{remark}
Hyperbolic space has a large group of isometries given by the Lie group $SO(1,n-1)$ of linear transformations in $\R^n$ that leave invariant the Lorentzian metric of Minkowski spacetime in dimension $n+1$. The Lie algebra of Killing fields on hyperbolic space induces several conserved quantities along the geodesic flow that can be used to show the \emph{complete integrability} of the geodesic flow in $\H^n$. 
\end{remark}

\subsection{Estimates for the geodesic flow}

The decay estimate for the spatial density is proven by obtaining time decay of the velocity support for the corresponding Vlasov fields. We first show estimates in time for the velocity coordinates along the geodesics in $\supp(f)$.

\begin{lemma}
The radial coordinate $r(t)$ along a geodesic $\gamma$ such that $\gamma(0)\in \supp(f_0)$ satisfies that for every $t\geq 0$,
\begin{equation}\label{sinh_bounds} ce^{Et}\leq \sinh r(t)\leq  C e^{Et},\end{equation}
where $c,C$ are constants depending only on $\supp(f_0)$. 
\end{lemma}

\begin{proof}
The geodesic equation for the radial coordinate can be written as $$\dfrac{\mathrm{d}v^r}{\mathrm{d}t}=\coth r\dfrac{l^2}{\sinh^2 r},$$ by using the angular velocity $l$. As a result, the first coordinate $\cosh r$ of a geodesic on the hyperboloid $\H^2$ satisfies the linear ode $$\dfrac{\mathrm{d}^2}{\mathrm{d}t^2}\cosh r=\dfrac{\mathrm{d}}{\mathrm{d}t}(\sinh r v^r)=\cosh r (v^r)^2+\cosh r \dfrac{l^2}{\sinh^2 r}=E^2\cosh r, $$ which can be integrated explicitly. The radial terms given by $$c_{\pm}(t):=\cosh r(t)\pm \dfrac{1}{E}\sinh r(t) v^r(t)$$ satisfy the linear odes $$\dfrac{\mathrm{d}}{\mathrm{d}t}c_{\pm}(t)=\sinh r v^r\pm E\cosh r=\pm E\Big(\cosh r \pm \dfrac{1}{E}\sinh r v^r\Big)=\pm E c_{\pm}(t),$$ and consequently $$\cosh r(t)\pm \dfrac{1}{E}\sinh r(t) v^r(t)=\Big(\cosh r(0)\pm \dfrac{1}{E}\sinh r(0) v^r(0)\Big)e^{\pm Et}.$$ In particular, the first coordinate $\cosh r$ of a geodesic on the hyperboloid $\H^2$ is given by
\begin{align}
\begin{aligned}\label{integrated_eqn_first_coordinate_hyperboloid}
    \cosh r(t)=~& c_+(0)e^{Et}+c_-(0)e^{-Et}\\
    =~& \dfrac{e^{Et}}{2}\Big(\cosh r(0)+ \dfrac{1}{E}\sinh r(0) v^r(0)\Big)+\dfrac{e^{-Et}}{2}\Big(\cosh r(0)- \dfrac{1}{E}\sinh r(0) v^r(0)\Big),
\end{aligned}
\end{align}
in terms of the radial coordinates $r(0)$ and $v^r(0)$.

We observe that the identity \eqref{identity_relate_particle_energy_angular_velocity_1} relating the particle energy $E$ and the angular velocity $l$ can be written as $$c_+(t)c_-(t)=\cosh^2 r(t)-\dfrac{1}{E^2} \sinh^2r(t) v^r(t)^2=1+\dfrac{l^2}{E^2},$$ which shows that the terms $c_\pm(0)$ are strictly positive on the initial data, see \eqref{integrated_eqn_first_coordinate_hyperboloid}. In particular, there are positive constants $\e,M$ such that $M\ge c_\pm(0)\ge \e,$ for all initial data in the support of $f_0$. Hence, for any initial data in the support of the initial distribution function we bound  the radial coordinate along the geodesic flow $$\log(\e e^{Et}) \leq
    r(t)=\log\Big(c_+(0)e^{Et}+c_-(0)e^{-Et}+\Big((c_+(0)e^{Et}+c_-(0)e^{-Et})^2-1\Big)^{\frac{1}{2}}\Big)\leq \log(4Me^{Et}),$$ by applying the inverse of hyperbolic cosine on the identity \eqref{integrated_eqn_first_coordinate_hyperboloid} combined with the bounds on $c_\pm(0)$. In particular, the radial coordinate along any geodesics emanating from the support of the initial distribution function satisfies 
\begin{equation}\label{sinh_bounds_2} ce^{Et}\leq \sinh r(t)\leq  C e^{Et},\end{equation}
where $c,C$ are positive constants depending only on the support of the initial distribution function. 
\end{proof}

We proceed to estimate the velocity coordinates of an arbitrary geodesic on hyperbolic space with initial data in the support of the initial distribution function, by exploiting the estimate \eqref{sinh_bounds} on the radial coordinate $r(t)$. 

\begin{proposition}
The velocity variables along a geodesic $\gamma$ such that $\gamma(0)\in \supp(f_0)$ satisfy that for every $t\geq 0$,
\begin{equation}\label{estimate_velocity_coordinate_hyperbolic_space_fix_particle_energy_hyp}
 |v^{\theta}\sinh r(t)|\leq \dfrac{L}{c\exp{(Et)}},\qquad E^2-(v^r)^2\leq \dfrac{L^2}{c^2\exp{(2Et)}},
\end{equation} 
where $c>0$ is a constant depending only on $\supp(f_0)$. 
\end{proposition}

\begin{proof}
By the compact support assumption, the absolute value of the angular velocity $|l|$ is uniformly bounded by a constant $L$, among all geodesics determined by the support of the initial distribution. We obtain that the angular velocity coordinate decays exponentially $$|v^{\theta}\sinh r(t)|=\dfrac{|l|}{\sinh r(t)}\leq \dfrac{L}{c\exp{(Et)}},$$ by using the definition of $l$. Furthermore, the radial velocity coordinate converges exponentially to the particle energy $$E^2-(v^r)^2=\dfrac{l^2}{\sinh^2 r}\leq \dfrac{L^2}{c^2\exp{(2Et)}}, $$ by using the identity \eqref{identity_relate_particle_energy_angular_velocity_1} relating the particle energy and the angular velocity. 
\end{proof}

In particular, for every Vlasov field initially supported on $\D_{\alpha}$, we obtain a uniform exponential decay estimate among all geodesics emanating from the support of the initial distribution.

\begin{corollary}
Let $\alpha>0$. Let $f_0$ be a regular initial data for the Vlasov equation on hyperbolic space that is compactly supported on $\D_{\alpha}$. The velocity variables along a geodesic $\gamma$ such that $\gamma(0)\in \supp(f_0)$ satisfy that for every $t\geq 0$,
\begin{equation}\label{estimate_velocity_coordinate_hyperbolic_space}
      |v^{\theta}\sinh r(t)|\leq \dfrac{L}{c\exp(\alpha t)},\qquad  \bigg(E^2- \dfrac{L^2}{c^2\exp{(2\alpha t)}}\bigg)^{\frac{1}{2}}\leq |v^r(t)|\le E,   
\end{equation}
where $c>0$ is a constant depending only on $\supp(f_0)$. 
\end{corollary}

\begin{proof}
It follows directly by using that $E$ is bounded below by $\alpha$ among all the geodesics determined by the support of the initial distribution.
\end{proof}

\subsection{Decay for Vlasov fields supported on $\D_{\alpha}$}\label{subsection_decay_derivatives_spatial_density_dimension_two} In this subsection, we prove exponential decay of the spatial density induced by a Vlasov field on $\H^n$ compactly supported on $\D_{\alpha}$.

\begin{proof}[Proof of the estimate (\ref{estimate_spatial_density_no_derivatives})]
Let us first introduce some notation. Given $x\in \H^2$ and $t\geq 0$ we define $$\Omega(t,x):=\Big\{v\in T_x\H^2:f(t,x,v)\ne 0\Big\}.$$ Observe that the set $\Omega(t,x)$ can be considered as the domain of integration of the spatial density. Since $f(t,x,v)=f_0(\phi_{-t}(x,v))$, we have that
\begin{equation}\label{Omega_2}
    \Omega(t,x)=\Big\{v\in T_x\H^2: \text{there exists $(x',v')\in \supp(f_0)$ such that $\phi_t(x',v')=(x,v)$} \Big\},
\end{equation}
and that $\|f(t)\|_{L^{\infty}_{x,v}}=\|f(0)\|_{L^{\infty}_{x,v}}$.

We proceed to estimate the spatial density $\rho(f)(t,x)$. Fix $x=(r,\theta)\in\H^2$ and $t\ge 0$. It follows from \eqref{estimate_velocity_coordinate_hyperbolic_space} and \eqref{Omega_2} that if $(v^\theta,v^r)\in \Omega(t,x)$, then 
\begin{equation}\label{equation_decay_ptheta}
      |v^{\theta}\sinh r|\leq \dfrac{L}{c\exp( \alpha t)},\qquad  \bigg(E^2- \dfrac{L^2}{c^2\exp{(2\alpha t)}}\bigg)^{\frac{1}{2}}\leq |v^r|\le E,
\end{equation}
and therefore 
\begin{equation}\label{decay_mom_support_hyp}
\vol_{T_x\H^2}(\Omega(t,x))=  \int_{\Omega(t,x)}\sinh r\mathrm{d}v^r\mathrm{d}v^{\theta} \le \frac{D_0}{\exp(\alpha t)},
\end{equation}
where $D_0$ is a positive constant depending on the support of $f_0$. Finally, $$|\rho(f)(t,x)|=\Big|\int f(t,x,v)\sinh r\mathrm{d}v^r\mathrm{d}v^{\theta}\Big|\leq \|f_0\|_{L^{\infty}_{x,v}} 
    \ \int_{\Omega(t,x)}\sinh r\mathrm{d}v^r\mathrm{d}v^{\theta} 
    \lesssim   \dfrac{\|f_0\|_{L^{\infty}_{x,v}}}{\exp(\alpha t)}.$$
\end{proof}

Next, we prove decay estimates for derivatives of the spatial density induced by Vlasov fields supported on $\D_{\alpha}$. We consider the explicit parallel frame $\{\dot{\gamma}, N\}$ set in \eqref{frame_parallel_orthogonal_dimension_two_hyperbolic_space}. We observe that the frame $\{\partial_r,(\sinh r)^{-1} \partial_{\theta}\}$ can be written in terms of the parallel frame $\{\dot{\gamma},N\}$ as $$\partial_r =\dfrac{v^r}{E^2}\dot{\gamma}-\dfrac{\sinh r v^{\theta}}{E^2}N,\qquad \dfrac{\partial_{\theta}}{\sinh r}=\dfrac{\sinh r v^{\theta}}{E^2}\dot{\gamma}+\dfrac{v^r}{E^2}N.$$ In the following, we combine the class of commuting vector fields $\lambda$ for the Vlasov equation on hyperbolic space, with the decay in time of the velocity support of the distribution function.  We start with a remark.

\begin{remark}\label{rem_for_derivatives}
Let $Z\in\lambda$. Note that $\supp(Zf_0)\subset \supp(f_0)$, and that since $Zf$ is a solution to the Vlasov equation (see Lemma \ref{lemma_commuting_vector_fields_hyperboli_dimension_two}), we have that $Zf(t,x,v)=Zf_0(\phi_{-t}(x,v))$. Then 
\begin{equation}\label{remark_support_lambda}
\|Zf(t)\|_{L^{\infty}_{x,v}}=\|Zf(0)\|_{L^{\infty}_{x,v}}\quad \text{ and
} \quad   \{v\in T_x\H^2:Zf(t,x,v)\ne 0\}\subset \Omega(t,x).
\end{equation}
\end{remark}

\begin{proof}[Proof of the estimate (\ref{estimate_spatial_density_with_derivatives})]
Since every spatial derivative of the spatial density is equal to the spatial density of the corresponding horizontal derivative of the distribution function, we have that
\begin{align}
\begin{aligned}\label{identity_derivatives_spatial_dens_intro_horiz_notat}
    \partial_{r}\rho(f)(t,x)&= \rho(\Hor_{(x,v)}(\partial_{r})f)(t,x)=:\rho(H_r f)(t,x),\\
    \dfrac{1}{\sinh r}\partial_{\theta}\rho(f)(t,x)&= \rho\Big(\Hor_{(x,v)}\Big(\dfrac{\partial_{\theta}}{\sinh r}\Big)f\Big)(t,x)=: \rho(H_{\theta}f)(t,x),
\end{aligned}
\end{align}
where the horizontal vector fields $H_r$ and $H_{\theta}$ can be written as 
\begin{equation}\label{identity_decomposition_horizontal_frame_dimension_two_hyp}
    H_r =\dfrac{v^r}{E^2}X-\dfrac{\sinh r v^{\theta}}{E^2}H,\qquad H_{\theta}=\dfrac{\sinh r v^{\theta}}{E^2}X+\dfrac{v^r}{E^2}H,
\end{equation}
in terms of the vector fields $H$ and $X$. According to the proof of the estimate \eqref{estimate_spatial_density_no_derivatives}, see  \eqref{estimate_velocity_coordinate_hyperbolic_space}, the weight $v^r$ converges exponentially to the particle energy $E$, whereas the weight $\sinh r v^{\theta}$ decays exponentially in time.

\textbf{Estimate for the derivative $(\sinh r)^{-1}\partial_{\theta}\rho(f)$.} By the previous decomposition of $H_{\theta}$, the angular derivative $(\sinh r)^{-1}\partial_{\theta}\rho(f)$ can be decomposed into
$$\dfrac{1}{\sinh r}\partial_{\theta}\rho(f)(t,x)= \rho\Big(\dfrac{\sinh r v^{\theta}}{E^2}Xf\Big)(t,x)+\rho\Big(\dfrac{v^r}{E^2}Hf\Big)(t,x)=:A_1+B_1.$$

On the one hand, the first term $A_1$ can be written using the commuting vector field $tX+Y$ by $$A_1=\dfrac{1}{t}\int \dfrac{\sinh r v^{\theta}}{E^2}(tX+Y)f \dvol_{T_x\H^2}(v)-\dfrac{1}{t}\int \dfrac{\sinh r v^{\theta}}{E^2}Yf \dvol_{T_x\H^2}(v),$$
where the second term in the RHS can be integrated by parts
\begin{align*}
    \int \dfrac{\sinh r v^{\theta}}{E^2}Yf \dvol_{T_x\H^2}(v)&=\int \dfrac{\sinh r v^{\theta}v^r}{E^2}\partial_{v^r}f \dvol_{T_x\H^2}(v)+\int \dfrac{\sinh r (v^{\theta})^2}{E^2}\partial_{v^{\theta}}f \dvol_{T_x\H^2}(v)\\
    &=-\int \dfrac{\sinh r v^{\theta}}{E^4}(\sinh^2 r (v^{\theta})^2-(v^r)^2)f \dvol_{T_x\H^2}(v)\\
    &\qquad-\int \dfrac{2\sinh r v^{\theta}(v^r)^2}{E^4}f \dvol_{T_x\H^2}(v)\\
    &=-\int \dfrac{\sinh r v^{\theta}}{E^2}f \dvol_{T_x\H^2}(v).
\end{align*}
We obtain that
$$A_1=\dfrac{1}{t}\int \dfrac{\sinh r v^{\theta}}{E^2}(tX+Y)f \dvol_{T_x\H^2}(v)+\dfrac{1}{t}\int \dfrac{\sinh r v^{\theta}}{E^2}f \dvol_{T_x\H^2}(v).$$ Moreover, we can make use of the decay of the velocity support of the distribution function, see \eqref{decay_mom_support_hyp}, together with \eqref{equation_decay_ptheta} and \eqref{remark_support_lambda} to obtain that
\begin{align*}
    |A_1|&\leq 
\Big|\dfrac{1}{t}\int \dfrac{\sinh r v^{\theta}}{E^2}(tX+Y)f \dvol_{T_x\H^2}(v)\Big|+\Big| \dfrac{1}{t}\int \dfrac{\sinh r v^{\theta}}{E^2}f \dvol_{T_x\H^2}(v)\Big|\\
&\leq \dfrac{1}{t}\|Yf_0\|_{L^{\infty}_{x,v}} \int_{\Omega(t,x)} \dfrac{|\sinh r v^{\theta}|}{E^2} \dvol_{T_x\H^2}(v)+ \dfrac{1}{t}\|f_0\|_{L^{\infty}_{x,v}}\int_{\Omega(t,x)} \dfrac{|\sinh r v^{\theta}|}{E^2} \dvol_{T_x\H^2}(v)\\
&\lesssim \dfrac{1}{te^{2\alpha t}}\|f_0\|_{W^{1,\infty}_{x,v}},
\end{align*}
where in the second inequality we used \eqref{remark_support_lambda} and in the third estimate we used \eqref{equation_decay_ptheta}. 

On the other hand, the second term $B_1$ can be written using the unstable vector field $U$ by
\begin{align*}
    B_1&=\rho\Big(\dfrac{v^r}{E^2}Hf\Big)(t,x)\\
    &=\int \dfrac{v^r}{E^2}(H+EV)f\dvol_{T_x\H^2}(v)-\int \dfrac{v^r}{E}Vf \dvol_{T_x\H^2}(v),\\
    &=\int \dfrac{1}{e^{Et}}\dfrac{v^r}{E^2}Uf\dvol_{T_x\H^2}(v)-\int\dfrac{v^r}{E}Vf \dvol_{T_x\H^2}(v),
\end{align*}
where the second term in the RHS can be integrated by parts
\begin{align*}
    \int\dfrac{v^r}{E}Vf \dvol_{T_x\H^2}(v)&=\int\dfrac{(v^r)^2}{E}\dfrac{\partial_{v^{\theta}}f}{\sinh r} \dvol_{T_x\H^2}(v)-\int\dfrac{\sinh rv^{\theta}v^r}{E}\partial_{v^r}f \dvol_{T_x\H^2}(v)\\
    &=\int\dfrac{\sinh r v^{\theta}(v^r)^2}{E^3}f \dvol_{T_x\H^2}(v)+\int\dfrac{\sinh^3 r(v^{\theta})^3}{E^3}f \dvol_{T_x\H^2}(v)\\
    &=\int\dfrac{\sinh r v^{\theta}}{E}f \dvol_{T_x\H^2}(v).
\end{align*}
We obtain that 
$$B_1=\int \dfrac{1}{e^{Et}}\dfrac{v^r}{E^2}Uf\dvol_{T_x\H^2}(v)-\int\dfrac{\sinh r v^{\theta}}{E}f \dvol_{T_x\H^2}(v).$$ Similarly to the estimate of $A_1$, we estimate $B_1$ by combining Remark \ref{rem_for_derivatives}, the decay of $\sinh r v^\theta$ and the decay of the velocity support of the distribution function
\begin{align*}
    |B_1|&\leq 
\Big|\int \dfrac{1}{e^{Et}}\dfrac{v^r}{E^2}Uf\dvol_{T_x\H^2}(v)\Big|+\Big| \int\dfrac{\sinh r v^{\theta}}{E}f \dvol_{T_x\H^2}(v)\Big|\\
& \leq\dfrac{1}{e^{\alpha t}}\|Uf_0\|_{L^{\infty}_{x,v}} \int_{\Omega(t,x)} \dfrac{|v^r|}{E^2}\dvol_{T_x\H^2}(v)+\|f_0\|_{L^{\infty}_{x,v}}\int_{\Omega(t,x)}\dfrac{|\sinh r v^{\theta}|}{E} \dvol_{T_x\H^2}(v)\\
&\lesssim \dfrac{1}{e^{2\alpha t}}\|f_0\|_{W^{1,\infty}_{x,v}}.
\end{align*}
Therefore, we obtain that the angular derivative of the spatial density $(\sinh r)^{-1}\partial_{\theta}\rho(f)$ is bounded above by
$$|(\sinh r)^{-1}\partial_{\theta}\rho(f)|\leq |A_1|+|B_1|\lesssim \|f_0\|_{W^{1,\infty}_{x,v}}e^{-{2\alpha t}}.$$ 

\textbf{Estimate for the derivative $\partial_{r}\rho(f)$.} By the previous decomposition of $H_r$, the radial derivative $\partial_{r}\rho(f)$ can be decomposed into
$$\partial_{r}\rho(f)(t,x)= \rho\Big( \dfrac{v^r}{E^2}Xf\Big)(t,x)-\rho\Big(\dfrac{\sinh r v^{\theta}}{E^2}Hf\Big)(t,x)=:A_2-B_2.$$

First, we write $A_2$ using the commuting vector field $tX+Y$ by $$A_2=\rho\Big( \dfrac{v^r}{E^2}Xf\Big)(t,x)=\dfrac{1}{t}\int \dfrac{v^r}{E^2}(tX+Y)f \dvol_{T_x\H^2}(v)-\dfrac{1}{t}\int \dfrac{v^r}{E^2}Yf \dvol_{T_x\H^2}(v),$$ where the second term in the RHS can be integrated by parts
\begin{align*}
    \int \dfrac{v^r}{E^2}Yf &\dvol_{T_x\H^2}(v)=\int \dfrac{(v^r)^2}{E^2}\partial_{v^r}f \dvol_{T_x\H^2}(v)+\int \dfrac{v^rv^{\theta}}{E^2}\partial_{v^{\theta}}f \dvol_{T_x\H^2}(v)\\
    &=-\int \dfrac{2v^r(v^{\theta})^2\sinh^2 r}{E^4}f \dvol_{T_x\H^2}(v)+\int \dfrac{\sinh^2 r v^r(v^{\theta})^2-(v^r)^3}{E^4}f \dvol_{T_x\H^2}(v)\\
    &=-\int \dfrac{v^r}{E^2}f \dvol_{T_x\H^2}(v),
\end{align*}
to obtain that $$A_2=\dfrac{1}{t}\int \dfrac{v^r}{E^2}(tX+Y)f \dvol_{T_x\H^2}(v)+\dfrac{1}{t}\int \dfrac{v^r}{E^2}f \dvol_{T_x\H^2}(v).$$ We estimate $A_2$ by making use of the decay of the velocity support of the distribution function and Remark \ref{rem_for_derivatives}.
\begin{align*}
    |A_2|&\leq 
\Big|\dfrac{1}{t}\int \dfrac{v^r}{E^2}(tX+Y)f \dvol_{T_x\H^2}(v) \Big|+\Big| \dfrac{1}{t}\int \dfrac{v^r}{E^2}f \dvol_{T_x\H^2}(v)\Big|\\
&\leq  \dfrac{1}{t}\|Yf_0\|_{L^{\infty}_{x,v}}\int_{\Omega(t,x)} \dfrac{|v^r|}{E^2} \dvol_{T_x\H^2}(v) + \dfrac{1}{t}\|f_0\|_{L^{\infty}_{x,v}}\int_{\Omega(t,x)} \dfrac{|v^r|}{E^2} \dvol_{T_x\H^2}(v)\\
&\lesssim \dfrac{1}{te^{\alpha t}}\|f_0\|_{W^{1,\infty}_{x,v}}.
\end{align*}

Similarly, the second term $B_2$ can be written using the unstable vector field $U$ by
\begin{align*}
    B_2&=\rho\Big(\dfrac{\sinh r v^{\theta}}{E^2}Hf\Big)(t,x)\\
    &=\int \dfrac{\sinh r v^{\theta}}{E^2}(H+EV)f\dvol_{T_x\H^2}(v)-\int \dfrac{\sinh r v^{\theta}}{E}Vf \dvol_{T_x\H^2}(v)\\
    &=\int \dfrac{1}{e^{Et}}\dfrac{\sinh r v^{\theta}}{E^2}Uf\dvol_{T_x\H^2}(v)-\int \dfrac{\sinh r v^{\theta}}{E}Vf \dvol_{T_x\H^2}(v),
\end{align*}
where the second term in the RHS can be integrated by parts
\begin{align*}
    \int \dfrac{\sinh r v^{\theta}}{E}Vf \dvol_{T_x\H^2}(v)&=\int \dfrac{ v^{\theta}v^{r}}{E}\partial_{v^{\theta}}f \dvol_{T_x\H^2}(v)-\int \dfrac{\sinh^2 r (v^{\theta})^2}{E}\partial_{v^r}f \dvol_{T_x\H^2}(v)\\
    &=-\int \dfrac{ (v^{r})^3}{E^3}f \dvol_{T_x\H^2}(v)-\int \dfrac{\sinh^2 r (v^{\theta})^2v^r}{E^3}f \dvol_{T_x\H^2}(v)\\
    &=-\int \dfrac{ v^{r}}{E}f \dvol_{T_x\H^2}(v),
\end{align*}
to obtain that $$B_2=\int \dfrac{1}{e^{Et}}\dfrac{\sinh r v^{\theta}}{E^2}Uf\dvol_{T_x\H^2}(v)+\int \dfrac{ v^{r}}{E}f \dvol_{T_x\H^2}(v).$$ We estimate $B_2$ combining \eqref{equation_decay_ptheta}, \eqref{decay_mom_support_hyp} and  \eqref{remark_support_lambda}.

\begin{align*}
    |B_2|&\leq 
\Big|\int \dfrac{1}{e^{Et}}\dfrac{\sinh r v^{\theta}}{E^2}Uf\dvol_{T_x\H^2}(v) \Big|+\Big| \int \dfrac{ v^{r}}{E}f \dvol_{T_x\H^2}(v)\Big|\\
&\leq \dfrac{1}{e^{\alpha t}}\|Uf_0\|_{L^{\infty}_{x,v}}\int_{\Omega(t,x)} \dfrac{|\sinh r v^{\theta}|}{E^2}\dvol_{T_x\H^2}(v) + \|f_0\|_{L^{\infty}_{x,v}}\int_{\Omega(t,x)} \dfrac{ |v^{r}|}{E}\dvol_{T_x\H^2}(v)\\
&\lesssim \dfrac{1}{e^{3\alpha t}}\|f_0\|_{W^{1,\infty}_{x,v}}+\dfrac{1}{e^{\alpha t}}\|f_0\|_{W^{1,\infty}_{x,v}}\\
&\lesssim \dfrac{1}{e^{\alpha t}}\|f_0\|_{W^{1,\infty}_{x,v}}.
\end{align*}
Finally, we obtain that the radial derivative of the spatial density $\partial_{r}\rho(f)$ is bounded above by
$$|\partial_{r}\rho(f)|\leq |A_2|+|B_2|\lesssim \|f_0\|_{W^{1,\infty}_{x,v}}e^{-{\alpha t}}.$$ 
\end{proof}

\subsection{Decay for compactly supported Vlasov fields}\label{subsection_decay_compact_support_vlasov_hyperbolic}

In this subsection, we prove exponential decay of the spatial density induced by a compactly supported Vlasov field on $\H^n$.

\begin{proof}[Proof of the estimate (\ref{estimate_spatial_density_no_derivatives_no_suppt_restrict})]
Given $x\in \H^2$ and $t\geq 1$, we recall the notation $$\Omega(t,x):=\Big\{v\in T_x\H^2:f(t,x,v)\ne 0\Big\}.$$ Fix $x=(r,\theta)\in\H^2$ and $t\ge 1$. It follows from \eqref{estimate_velocity_coordinate_hyperbolic_space_fix_particle_energy_hyp} that if $(v^\theta,v^r)\in \Omega(t,x)$, then 
\begin{equation}\label{equation_decay_ptheta_slow}
      |v^{\theta}\sinh r|\leq \dfrac{L}{c\exp( E t)},\qquad  \bigg(E^2- \dfrac{L^2}{c^2\exp{(2E t)}}\bigg)^{\frac{1}{2}}\leq |v^r|\le E.
\end{equation}
Let us consider polar coordinates in the fibers of the tangent bundle by $$E:=\sqrt{(v^r)^2+\sinh^2 r (v^{\theta})^2},\qquad \varphi:=\arctan\Big(\frac{v^r}{v^{\theta}\sinh r}\Big).$$ For our purposes, we change variables in the integration in the velocity variables that sets the spatial density as 
\begin{equation}\label{formula_spatial_density_polar_coordinates}
\rho(f)(t,x)=\int f(t,x,v)\mathrm{d}v^r\mathrm{d}(v^{\theta}\sinh r)=\int f(t,E,\varphi)E\mathrm{d}E\mathrm{d}\varphi,
\end{equation}
where we used the change of variables $(v^r,v^{\theta})\mapsto (v^r,v^{\theta}\sinh r)$ in the first equality, and the change of variables $ (v^r,v^{\theta}\sinh r)\mapsto (E,\varphi)$ in the second equality. By the compact support assumption, there exists $E_{\max}>0$ such that for every $(x,v)\in\supp(f_0)$ we have $E\leq E_{\max}$. As a result, we have $E\leq E_{\max}$ for every $(x,v)\in\supp(f)$ by the conservation along the geodesic flow of the particle energy. Therefore, we have 
\begin{equation}\label{decay_mom_support_slow}
\vol_{T_x\H^2}(\Omega(t,x))=  \int_{\Omega(t,x)}E\mathrm{d}E\mathrm{d}\varphi \leq 2\pi\int_{0}^{E_{\max}}  \dfrac{E}{e^{Et}}\mathrm{d}E \lesssim \frac{1}{t^2}.
\end{equation}
Finally, we obtain $$ |\rho(f)(t,x)|=\Big|\int f(t,x,v)\dvol_{T_x\H^2}(v)\Big|\leq \|f_0\|_{L^{\infty}_{x,v}} 
    \vol_{T_x\H^2}(\Omega(t,x))
    \lesssim   \dfrac{1}{t^2}\|f_0\|_{L^{\infty}_{x,v}}.$$
\end{proof}

Next, we prove decay estimates for derivatives of the spatial density induced by compactly supported Vlasov fields on $\H^2$. The proof of \eqref{estimate_spatial_density_with_derivatives_no_suppt_restrict} follows the same strategy as in the proof of \eqref{estimate_spatial_density_with_derivatives} in the previous subsection. Nonetheless, the estimate of every velocity average is performed slightly differently. Now, we need to keep track of the weights in $E$ which determine the time decay of every velocity average. The weights that decay exponentially are irrelevant due to the contribution of the particles with arbitrarily small particle energies.

\begin{proof}[Proof of the estimate (\ref{estimate_spatial_density_with_derivatives_no_suppt_restrict})] We quickly follow the strategy in the proof of the estimate \eqref{estimate_spatial_density_with_derivatives} under the appropriate modifications to obtain the correct decay rates in this case. Recall the formulae \eqref{identity_derivatives_spatial_dens_intro_horiz_notat} for the derivatives of the spatial density in terms of the corresponding horizontal vector fields, and the decomposition \eqref{identity_decomposition_horizontal_frame_dimension_two_hyp} for the horizontal vector fields in terms of $X$ and $H$. 

By the compact support assumption, there exists $E_{\max}>0$ such that for every $(x,v)\in\supp(f_0)$ we have $E\leq E_{\max}$. As a result, we have $E\leq E_{\max}$ for every $(x,v)\in\supp(f)$ by the conservation along the geodesic flow of the particle energy. According to the proof of the estimate \eqref{estimate_velocity_coordinate_hyperbolic_space_fix_particle_energy_hyp}, the weight $v^r$ converges exponentially to the particle energy $E$, whereas the weight $\sinh r v^{\theta}$ decays exponentially in time. Observe that we have uniform boundedness for $\frac{v^r}{E}$ and $\frac{v^{\theta}\sinh r}{E}$ by definition of the particle energy.

\textbf{Estimate for the derivative $(\sinh r)^{-1}\partial_{\theta}\rho(f)$.} By the decomposition of $H_{\theta}$ in terms of $X$ and $H$, the angular derivative $(\sinh r)^{-1}\partial_{\theta}\rho(f)$ can be decomposed into
$$\dfrac{1}{\sinh r}\partial_{\theta}\rho(f)(t,x)= \rho\Big(\dfrac{\sinh r v^{\theta}}{E^2}Xf\Big)(t,x)+\rho\Big(\dfrac{v^r}{E^2}Hf\Big)(t,x)=:A_1+B_1.$$

The first term $A_1$ can be written using the commuting vector field $tX+Y$ by $$A_1=\dfrac{1}{t}\int \dfrac{\sinh r v^{\theta}}{E^2}(tX+Y)f \dvol_{T_x\H^2}(v)+\dfrac{1}{t}\int \dfrac{\sinh r v^{\theta}}{E^2}f \dvol_{T_x\H^2}(v),$$ after integrating by parts in the velocity variables. Using the decay of the velocity support of the distribution function \eqref{decay_mom_support_slow}, together with \eqref{equation_decay_ptheta_slow} and \eqref{remark_support_lambda} we obtain 
\begin{align*}
    |A_1|&\leq 
\Big|\dfrac{1}{t}\int \dfrac{\sinh r v^{\theta}}{E^2}(tX+Y)f \dvol_{T_x\H^2}(v)\Big|+\Big| \dfrac{1}{t}\int \dfrac{\sinh r v^{\theta}}{E^2}f \dvol_{T_x\H^2}(v)\Big|\\
&\leq \dfrac{1}{t}\|Yf_0\|_{L^{\infty}_{x,v}} \int_{\Omega(t,x)} \dfrac{|\sinh r v^{\theta}|}{E^2} \dvol_{T_x\H^2}(v)+ \dfrac{1}{t}\|f_0\|_{L^{\infty}_{x,v}}\int_{\Omega(t,x)} \dfrac{|\sinh r v^{\theta}|}{E^2} \dvol_{T_x\H^2}(v)\\
&\lesssim \frac{1}{t}\int_{0}^{E_{\max}}  \dfrac{\mathrm{d}E}{e^{2Et}}\|f_0\|_{W^{1,\infty}_{x,v}}\\
&\lesssim \dfrac{1}{t^2}\|f_0\|_{W^{1,\infty}_{x,v}},
\end{align*}
where we have used the change of variables $(v^r,v^{\theta})\mapsto (E,\varphi)$ considered in \eqref{formula_spatial_density_polar_coordinates}.
where in the second inequality we used \eqref{remark_support_lambda} and in the third inequality we used \eqref{equation_decay_ptheta}.

On the other hand, the second term $B_1$ can be written using the unstable vector field $U$ by
$$B_1=\int \dfrac{1}{e^{Et}}\dfrac{v^r}{E^2}Uf\dvol_{T_x\H^2}(v)-\int\dfrac{\sinh r v^{\theta}}{E}f \dvol_{T_x\H^2}(v),$$ 
after integrating by parts in the velocity variables. Similarly to the estimate of $A_1$, we estimate $B_1$ by combining Remark \ref{rem_for_derivatives}, the decay of $\sinh r v^\theta$ and the decay of the velocity support of the distribution function
\begin{align*}
    |B_1|&\leq 
\Big|\int \dfrac{1}{e^{Et}}\dfrac{v^r}{E^2}Uf\dvol_{T_x\H^2}(v)\Big|+\Big| \int\dfrac{\sinh r v^{\theta}}{E}f \dvol_{T_x\H^2}(v)\Big|\\
& \leq\|Uf_0\|_{L^{\infty}_{x,v}} \int_{\Omega(t,x)} \dfrac{1}{e^{E t}} \dfrac{|v^r|}{E^2}\dvol_{T_x\H^2}(v)+\|f_0\|_{L^{\infty}_{x,v}}\int_{\Omega(t,x)}\dfrac{|\sinh r v^{\theta}|}{E} \dvol_{T_x\H^2}(v)\\
&\lesssim \int_{0}^{E_{\max}}  \dfrac{\mathrm{d}E}{e^{2Et}}\|f_0\|_{W^{1,\infty}_{x,v}}+\int_{0}^{E_{\max}}  \dfrac{E\mathrm{d}E}{e^{2Et}}\|f_0\|_{W^{1,\infty}_{x,v}}\\
&\lesssim \dfrac{1}{t}\|f_0\|_{W^{1,\infty}_{x,v}}.
\end{align*}
Therefore, we obtain that the angular derivative of the spatial density $(\sinh r)^{-1}\partial_{\theta}\rho(f)$ is bounded above by
$$|(\sinh r)^{-1}\partial_{\theta}\rho(f)|\leq |A_1|+|B_1|\lesssim \|f_0\|_{W^{1,\infty}_{x,v}}t^{-1}.$$ 
\textbf{Estimate for the derivative $\partial_{r}\rho(f)$.} By the decomposition of $H_{r}$ in terms of $X$ and $H$, the radial derivative $\partial_{r}\rho(f)$ can be decomposed into
$$\partial_{r}\rho(f)(t,x)= \rho\Big( \dfrac{v^r}{E^2}Xf\Big)(t,x)-\rho\Big(\dfrac{\sinh r v^{\theta}}{E^2}Hf\Big)(t,x)=:A_2-B_2.$$

First, we write $A_2$ using the commuting vector field $tX+Y$ by $$A_2=\dfrac{1}{t}\int \dfrac{v^r}{E^2}(tX+Y)f \dvol_{T_x\H^2}(v)+\dfrac{1}{t}\int \dfrac{v^r}{E^2}f \dvol_{T_x\H^2}(v),$$ after integrating by parts in the velocity variables. We estimate $A_2$ by making use of the decay of the velocity support of the distribution function and Remark \ref{rem_for_derivatives}.
\begin{align*}
    |A_2|&\leq 
\Big|\dfrac{1}{t}\int \dfrac{v^r}{E^2}(tX+Y)f \dvol_{T_x\H^2}(v) \Big|+\Big| \dfrac{1}{t}\int \dfrac{v^r}{E^2}f \dvol_{T_x\H^2}(v)\Big|\\
&\leq  \dfrac{1}{t}\|Yf_0\|_{L^{\infty}_{x,v}}\int_{\Omega(t,x)} \dfrac{|v^r|}{E^2} \dvol_{T_x\H^2}(v) + \dfrac{1}{t}\|f_0\|_{L^{\infty}_{x,v}}\int_{\Omega(t,x)} \dfrac{|v^r|}{E^2} \dvol_{T_x\H^2}(v)\\
&\lesssim \frac{1}{t}\int_{0}^{E_{\max}}  \dfrac{\mathrm{d}E}{e^{Et}}\|f_0\|_{W^{1,\infty}_{x,v}}\\
&\lesssim \dfrac{1}{t^2}\|f_0\|_{W^{1,\infty}_{x,v}}.
\end{align*}

Similarly, the second term $B_2$ can be written using the unstable vector field $U$ by $$B_2=\int \dfrac{1}{e^{Et}}\dfrac{\sinh r v^{\theta}}{E^2}Uf\dvol_{T_x\H^2}(v)+\int \dfrac{ v^{r}}{E}f \dvol_{T_x\H^2}(v),$$ after integrating by parts in the velocity variables. We estimate $B_2$ combining \eqref{equation_decay_ptheta_slow}, \eqref{decay_mom_support_slow} and  \eqref{remark_support_lambda}.
\begin{align*}
    |B_2|&\leq 
\Big|\int \dfrac{1}{e^{Et}}\dfrac{\sinh r v^{\theta}}{E^2}Uf\dvol_{T_x\H^2}(v) \Big|+\Big| \int \dfrac{ v^{r}}{E}f \dvol_{T_x\H^2}(v)\Big|\\
&\leq \|Uf_0\|_{L^{\infty}_{x,v}}\int_{\Omega(t,x)}\dfrac{1}{e^{E t}} \dfrac{|\sinh r v^{\theta}|}{E^2}\dvol_{T_x\H^2}(v) + \|f_0\|_{L^{\infty}_{x,v}}\int_{\Omega(t,x)} \dfrac{ |v^{r}|}{E}\dvol_{T_x\H^2}(v)\\
&\lesssim \int_{0}^{E_{\max}}  \dfrac{\mathrm{d}E}{e^{3Et}}\|f_0\|_{W^{1,\infty}_{x,v}}+\int_{0}^{E_{\max}}  \dfrac{E\mathrm{d}E}{e^{Et}}\|f_0\|_{W^{1,\infty}_{x,v}}\\
&\lesssim \dfrac{1}{t}\|f_0\|_{W^{1,\infty}_{x,v}}.
\end{align*}
Finally, we obtain that the radial derivative of the spatial density $\partial_{r}\rho(f)$ is bounded above by
$$|\partial_{r}\rho(f)|\leq |A_2|+|B_2|\lesssim \|f_0\|_{W^{1,\infty}_{x,v}}t^{-1}.$$
\end{proof}

\subsection{Decay for Vlasov fields in higher dimensions}\label{subsection_decay_spatial_density_higher_dimension}
Similar decay estimates for the spatial density induced by a Vlasov field on hyperbolic space $(\H^n,g_{\H^n})$, can be proven by using a commuting vector field approach akin to the one considered in dimension two. The spatial density induced by a Vlasov field $f$ on $(\H^n,g_{\H^n})$ is given by $$\rho(f)(t,x)=\int f(t,x,v)\sinh^{n-1} r \mathrm{d}v^r\dvol_{\S^{n-1}}(v),$$ in terms of the volume form on $(\H^n,g_{\H^n})$ in local coordinates. The decay of the velocity support of the distribution function can be proven using the same strategy performed in the two dimensional case. First, one studies the ode satisfied by the first coordinate $\cosh r$ of a geodesic on the hyperboloid $\H^n$ given by
$$\dfrac{\mathrm{d}^2}{\mathrm{d}t^2}\cosh r=E^2\cosh r,$$ in terms of the particle energy $E^2:=g_{\H^n}(\dot{\gamma},\dot{\gamma})$. Later, one can use an orthonormal parallel frame $\{\dot{\gamma},N_1,\dots, N_{n-1}\}$ to build commuting vector fields as in Section \ref{subsection_commutators_vlasov_hyperbolic}. Using the commuting vector fields $$\Hor_{(x,v)}(v),\quad t\Hor_{(x,v)}(v)+\Ver_{(x,v)}(v),\quad  e^{\pm Et}(\Hor_{(x,v)}(N_i)\pm E\Ver_{(x,v)}(N_i)),$$ we can estimate the derivatives of the spatial density $$\partial_{x^i}\rho(f)(t,x)=\rho(\Hor_{(x,v)}(\partial_{x^i})f)(t,x),$$ by decomposing every horizontal vector field $\Hor_{(x,v)}(\partial_{x^i})$ in terms of the commuting vector fields plus errors that can be controlled after integration by parts in the velocity variables. Decay estimates can finally be derived using bounds for the geodesic flow in the support of the distribution function. 

\section{Decay for Vlasov fields on asymptotically hyperbolic manifolds}\label{section_dispersive_perturbations}

In this section, we prove decay estimates for the spatial density induced by a Vlasov field on a non-trapping asymptotically hyperbolic Riemannian manifold $(\M,g)$. The geodesic flow in $(\M,g)$ is determined by the geodesic equations $$\dfrac{\mathrm{d}x^{i}}{\mathrm{d}t}=v^{i},\qquad \dfrac{\mathrm{d}v^{i}}{\mathrm{d}t}=-\Gamma^{i}_{jk}v^{j}v^{k},$$ in terms of the Christoffel symbols $\Gamma^{i}_{jk}$ of $(\M,g)$. The geodesic flow in $(\M,g)$ can also be viewed as a \emph{Hamiltonian flow} in the cotangent bundle $(T^*\M, \bar{g})$, where the corresponding Hamiltonian $H:T^*\M\to[0,\infty)$ is given by $$H(x,v):=\dfrac{1}{2}g^{ij}v_{i}v_{j}.$$ We define the \emph{particle energy} $E: T\M\to [0,\infty) $ of a geodesic $\gamma$ by $$E(x,v):=g_x(v,v)^{\frac{1}{2}}=\sqrt{g_{ij}v^{i}v^{j}}.$$ The particle energy $E$ is conserved along the geodesic flow. Note that the Hamiltonian $H$ can be written in terms of the particle energy as $H=\frac{1}{2}E^2$. We also define the \emph{angular velocity} $l: T\M\to [0,\infty)$ of a geodesic $\gamma$ by $$l(x,v):=g(\partial_{\theta},\dot{\gamma}).$$ The angular velocity $l$ is almost conserved along the geodesic flow since the vector field $\partial_{\theta}$ is Killing on hyperbolic space. 

\begin{lemma}\label{lemma_estimates_veloc_coord_and_almost_cons_quant}
Let $E_{\mathrm{max}}>0$. Let $(x,v)\in T\M$ such that $E(x,v)\leq E_{\max}$. Then, the velocity coordinates of a geodesic $\gamma$ with $\gamma(0)=x$ and $\dot{\gamma}(0)=v$ satisfy 
\begin{equation}\label{unifom_bound_components_geodesic_ah}
|v^r|\leq E_{\mathrm{max}}+1,\qquad |v^{\theta}\sinh r| \leq E_{\max}+1,
\end{equation}
for $r$ sufficiently large. Furthermore, the particle energy and the angular velocity satisfy 
\begin{equation}\label{estimate_particle_energy_and_angular_velocity_ah}
E^2=(v^r)^2+\dfrac{l^2}{\sinh^2 r}+O(e^{-\beta r}),\qquad l=v^{\theta}\sinh^{2} r+O(e^{-(\beta-1)r}).
\end{equation}
\end{lemma}

\begin{proof}
Since $(\M,g)$ is asymptotically hyperbolic, we have 
\begin{equation}\label{estimate_particle_energy_ah}
|E^2(x,v)-(v^r)^2-\sinh^2 r (v^{\theta})^2|=O(e^{-\beta r}).
\end{equation}
Using this bound combined with the conservation along the geodesic flow of $E$, we obtain the estimates \eqref{unifom_bound_components_geodesic_ah}. By definition of the angular velocity $$l=v^rg\Big(\frac{\partial_{\theta}}{\sinh r},\partial_r\Big)\sinh r +v^{\theta}\sinh^2 rg\Big(\frac{\partial_{\theta}}{\sinh r},\frac{\partial_{\theta}}{\sinh r}\Big)=v^{\theta}\sinh^2 r+O(e^{-(\beta-1) r}),$$ where we have used the estimates \eqref{unifom_bound_components_geodesic_ah}. In particular, we obtain $l^2=(v^{\theta})^2\sinh^{4} r+O(e^{-(\beta-2) r}),$ which combined with \eqref{estimate_particle_energy_ah} results in $$E^2=(v^r)^2+\dfrac{l^2}{\sinh^2 r}+O(e^{- \beta r}).$$ 
\end{proof}

We also show that the geodesic flow converges at infinity to the geodesic flow in hyperbolic space. First, we take a look at the geodesic equations on $(\M,g)$.

\begin{lemma}\label{lemma_geod_eqns_ah_mfld}
The geodesic flow $\phi_t(x,v)$ on $(\M,g)$ satisfies that for every $t\geq 0$, we have $$\frac{\mathrm{d}v^{\theta}}{\mathrm{d}t} =-2\coth r v^rv^{\theta}+O(e^{- \beta r}),\qquad  \frac{\mathrm{d}v^r}{\mathrm{d}t} =\cosh r \sinh r (v^{\theta})^2+O(e^{- \beta r}).$$
\end{lemma}

\begin{proof}
By definition of the Christoffel symbols, for every $i$, $j$, $k\in\{r,\theta\}$, we have $|\Gamma_{jk}^{i}(g)-\Gamma_{jk}^{i}(g_{\H^n})|=O(e^{- \beta r})$ on an asymptotically hyperbolic manifold. Furthermore, we have boundedness of the components of $\gamma$ by Lemma \ref{lemma_estimates_veloc_coord_and_almost_cons_quant}. The lemma is a straightforward application of these bounds on the geodesic equations in $(\M,g)$.
\end{proof}

\subsection{Estimates for the geodesic flow}

For our purposes, we first show estimates in time for the velocity coordinates along the geodesics in $\supp(f)$.

\begin{lemma}\label{lemma_basic_estimate_radial_coord_along_time}
Let $R_1>0$ be sufficiently large. The radial coordinate $r(t)$ along a geodesic $\gamma$ in $\M\setminus \overline{B(R_1)}$ such that $\gamma(0)\in \supp(f_0)$ satisfies that for every $t\geq 0$, $$ce^{Et}\leq \sinh r(t)\leq  C e^{Et},$$ where $c,C$ are positive constants depending only on $\supp(f_0)$. 
\end{lemma}

\begin{proof}

From the first equation in \eqref{estimate_particle_energy_and_angular_velocity_ah}, we have $v^r>\frac{\alpha}{2}>0$ in the far-away region of $(\M,g)$, for a positive constant $c'>0$. Thus, the radial coordinate along the flow satisfies $r(t)\geq r(0)+\frac{\alpha}{2} t$. By Cauchy stability of the geodesic flow, we have $r(t)\geq 1+\frac{\alpha}{2} t$ for all geodesics in $\supp(f)$. We proceed to improve this elementary estimate by using the radial geodesic equation.

The geodesic equation for the radial coordinate can be written in terms of the angular velocity by
\begin{equation}\label{eq:dv^r_ah}\dfrac{\mathrm{d}v^r}{\mathrm{d}t}=\coth r\dfrac{l^2}{\sinh^2 r}+O(e^{-\beta r}),\end{equation}
where we have used Lemma \ref{lemma_estimates_veloc_coord_and_almost_cons_quant} and Lemma \ref{lemma_geod_eqns_ah_mfld}. As a result, the quantity $\cosh r$ satisfies the linear ode $$\dfrac{\mathrm{d}^2}{\mathrm{d}t^2}\cosh r=\dfrac{\mathrm{d}}{\mathrm{d}t}(\sinh r v^r)=\cosh r (v^r)^2+\cosh r \dfrac{l^2}{\sinh^2 r}+O(e^{-(\beta-1) r})=E^2\cosh r+O(e^{-(\beta-1) r}). $$ The radial terms given by $$c_{\pm}(t):=\cosh r(t)\pm \dfrac{1}{E}\sinh r(t) v^r(t)$$ satisfy the linear odes 
\begin{align*}
 \dfrac{\mathrm{d}}{\mathrm{d}t}c_{\pm}(t)&=\sinh r v^r\pm E\cosh r+O(e^{-(\beta-1) r})=\pm E\Big(\cosh r \pm \dfrac{1}{E}\sinh r v^r\Big)+O(e^{-(\beta-1) r})\\
 &=\pm E c_{\pm}(t)+O(e^{-(\beta-1) r}).
\end{align*}
Rearranging this ode, we obtain $$\dfrac{\mathrm{d}}{\mathrm{d}t}\Big[c_{\pm}(t)e^{\mp Et}\Big]=e^{\mp Et}O(e^{-(\beta-1) r}),$$ so integrating, we have $$c_{\pm}(t)=e^{\pm Et}c_{\pm}(0)+e^{\pm Et}\int_0^t e^{\mp E\tau}O(e^{-(\beta-1) r(\tau)})\mathrm{d}\tau.$$ In particular, we have $c_+(t)= e^{Et}(c_+(0)+O(1))$ and $c_-(t)= e^{-Et}c_-(0)+O(1)$ by using the lower bound $r(t)\geq 1+\frac{\alpha}{2} t$ for all geodesics in $\supp(f)$. Furthermore, we have $$\cosh r(t)=\frac{1}{2}(c_+(t)+c_-(t))=\frac{1}{2}(e^{Et}(c_+(0)+O(1))+e^{-Et}c_-(0)+O(1)).$$ We observe that the identity in \eqref{estimate_particle_energy_and_angular_velocity_ah} relating the particle energy $E$ and the angular velocity $l$ implies $$c_+(t)c_-(t)=\cosh^2 r(t)-\dfrac{1}{E^2} \sinh^2r(t) v^r(t)^2=1+\dfrac{l^2}{E^2}+O(e^{-(\beta-2) r}),$$ which shows that the terms $c_\pm(0)$ are strictly positive on the initial data. In particular, there are positive constants $\e,M$ such that $M\ge c_\pm(0)\ge \e,$ for all initial data in the support of $f_0$. 

Hence, for any initial data in the support of the initial distribution function we bound  the radial coordinate along the geodesic flow $$\log(\e e^{Et}) \leq
    r(t)\leq \log(4Me^{Et}),$$ by applying the bounds on $c_\pm$. In particular, the radial coordinate along any geodesics emanating from the support of the initial distribution function satisfies $$ce^{Et}\leq \sinh r(t)\leq  C e^{Et},$$ where $c,C$ are positive constants depending only on the support of the initial distribution function. 
\end{proof}

As a corollary, we obtain uniform boundedness in time of the angular velocity $l$ along an arbitrary geodesic with initial data in the support of the initial distribution function.

\begin{corollary}\label{cor_unif_bound_angular_velocity}
Let $R_1>0$ be sufficiently large. The angular velocity $l(t)$ along a geodesic $\gamma$ in $\M\setminus \overline{B(R_1)}$ such that $\gamma(0)\in \supp(f_0)$ satisfies that for every $t\geq 0$, $$|l(t)|\leq L,$$ where $L>0$ depends only on $\supp(f_0)$.
\end{corollary}

\begin{proof}
By Lemma \ref{lemma_geod_eqns_ah_mfld}, we have $\frac{\mathrm{d}}{\mathrm{d}t}(v^{\theta}\sinh^2 r )=O(e^{-(\beta-2) r}).$ The corollary follows by using Lemma \ref{lemma_basic_estimate_radial_coord_along_time} to integrate this ode in time.
\end{proof}

We proceed to estimate the velocity coordinates of an arbitrary geodesic on $\M$ with initial data in the support of the initial distribution function.

\begin{proposition}
Let $R_1>0$ be sufficiently large. The velocity variables along a geodesic $\gamma$ in $\M\setminus \overline{B(R_1)}$ such that $\gamma(0)\in \supp(f_0)$ satisfy that for every $t\geq 0$, $$ |v^{\theta}\sinh r(t)|\leq \dfrac{C}{\exp{(Et)}},\qquad E^2-(v^r)^2\leq \dfrac{C}{\exp{(2Et)}},$$ where $C>0$ is a constant depending only on $\supp(f_0)$. 
\end{proposition}

\begin{proof}
By the compact support assumption of the initial distribution and Corollary \ref{cor_unif_bound_angular_velocity}, the angular velocity satisfies $|l(t)|\leq L$ among all geodesics determined by the support of the initial distribution. We obtain that the angular velocity coordinate decays exponentially $$|\sinh rv^{\theta}(t)|=\dfrac{|l|}{\sinh r(t)}+O(e^{-\beta r})\leq \dfrac{C}{\exp{(Et)}}, $$ by using Lemma \ref{lemma_estimates_veloc_coord_and_almost_cons_quant}. Furthermore, we obtain that the radial velocity coordinate converges exponentially to the particle energy $$E^2-(v^r)^2=\dfrac{l^2}{\sinh^2 r}+O(e^{-\beta r})\leq \frac{C}{\exp{(2Et)}},$$ by using the identity relating the particle energy and the angular velocity in Lemma \ref{lemma_estimates_veloc_coord_and_almost_cons_quant}. 
\end{proof}

In particular, for every Vlasov field initially supported on $\D_{\alpha}$, we obtain a uniform exponential decay estimate among all geodesics emanating from the support of the initial distribution function.

\begin{corollary}
Let $\alpha>0$. Let $R_1>0$ be sufficiently large. Let $f_0$ be a regular initial data for the Vlasov equation on hyperbolic space that is compactly supported on $\D_{\alpha}$. The velocity variables along a geodesic $\gamma$ in $\M\setminus \overline{B(R_1)}$ such that $\gamma(0)\in \supp(f_0)$ satisfy that for every $t\geq 0$,
\begin{equation}\label{estimate_velocity_coordinate_ah_space}
      |v^{\theta}\sinh r(t)|\leq \dfrac{C}{\exp(\alpha t)},\qquad  \bigg(E^2- \dfrac{C}{\exp{(2\alpha t)}}\bigg)^{\frac{1}{2}}\leq |v^r(t)|\le E,   
\end{equation}
where $C>0$ is a constant depending only on $\supp(f_0)$. 
\end{corollary}

\subsection{Decay for Vlasov fields supported on $\D_{\alpha}$}\label{subsection_decay_spatial_density_AH}

In this subsection, we prove a general dispersion estimate for the spatial density induced by a Vlasov field on a Riemannian manifold $\M$ under the assumptions of Theorem \ref{thm_decay_spatial_density_ah_manif}. We start by introducing some notation. Let $(\M,g)$ be asymptotically hyperbolic and non-trapping. By definition, there exists a compact set $K\subset \M$, $R_0>0$ and a diffeomorphism $\Psi:\H^n\setminus \overline{B(R_0)}\to \M\setminus K$, satisfying the conditions in Definition \ref{ahyp}. For $r>R_0$ we define $K_r:=K\cup \Psi(\overline{B(r)}\setminus \overline{B(R_0)})$, and $\widehat{K}_r:=\pi^{-1}(K_r)$, where we recall that $\pi:T\M\to\M$ is the canonical projection. Since $(\M,g)$ is asymptotically hyperbolic, as $r\to \infty$, the spaces $(\M\setminus K_r,g)$ and $(\H^n\setminus \overline{B(r)}, g_{\H^n})$ are $C^2$-close under the map $\Psi$.

As in Subsection \ref{subsection_decay_derivatives_spatial_density_dimension_two}, we introduce the domain of integration of the spatial density. For $x\in \M$ and $t\geq 0$ we define 
$$\Omega(t,x)=\Big\{v\in T_x\M:f(t,x,v)\ne 0\Big\}.$$
Since $f(t,x,v)=f_0(\phi_{-t}(x,v))$, we have that 
$$\Omega(t,x)=\Big\{v\in T_x\M: \text{there exists $(x',v')\in \supp(f_0)$ such that $\phi_t(x',v')=(x,v)$} \Big\}.$$
Observe that 
\begin{equation*}\label{eq:omega}
   \rho(f)(t,x)=\int_{\Omega(x,t)}f(t,x,v)\dvol_{T_x\M}(v). 
\end{equation*}

In order to prove the exponential decay in time of the spatial density, we proceed as in Subsection \ref{subsection_decay_derivatives_spatial_density_dimension_two} to establish the decay of the volume of the velocity support of the distribution function $\Omega(t,x)$. The following lemma states that if $(\M,g)$ is asymptotically hyperbolic and non-trapping, then geodesics escape to infinity uniformly. 

\begin{lemma}\label{lem:ahyp}
Let $(\M,g)$ be asymptotically hyperbolic and non-trapping, and $Q\subset \M$ a compact set. Assume that $f_0$ is compactly supported and its support lies on $\D_\alpha$, for some $\alpha>0$. Then there exists $T>0$ such that $\rho(f)(t,x)=0$, for every $x\in Q$ and $t\geq T$.
\end{lemma}

\begin{proof} From the identity \eqref{identity_relate_particle_energy_angular_velocity_1}, we have $v^r>c>0$ in the far-away region of hyperbolic space, for $c>0$. Similarly, from the first equation in \eqref{estimate_particle_energy_and_angular_velocity_ah}, we have $v^r>c>0$ in the far-away region of $(\M,g)$, for $c>0$. In other words, the same property holds for $\Phi^*g$-geodesics in $\H^n\setminus B(r)$ for sufficiently large $r$, and therefore, if a $\Psi^*g$-geodesic enters the region $\H^n\setminus B(r)$, it will not escape from it. Equivalently, if a $g$-geodesic enters $\M\setminus K_r$, it will stay in $\M\setminus K_r$. We choose $r$ large enough so that $Q$ and $\pi(\supp (f_0))$ are subsets of $K_r$. 

Since $(\M,g)$ is non-trapping, any non-zero vector in $\widehat{K_r}$ eventually enters $T\M\setminus \widehat{K_r}$. Define $T$ as the  minimum time that a vector in $\supp(f_0)\subset \widehat{K_r}$ needs to flow to enter $T\M\setminus \widehat{K_r}$; this is well defined since $\supp(f_0)$ is compact and a subset of $\D_\alpha$. Then, if $t\ge T$ and $(x',v')\in \supp(f_0)$ we have that $\phi_t(x',v')\in T\M\setminus \widehat{K_r}$. We conclude that if $t\geq T$ and $x\in Q$ $$\Omega(t,x)=\Big\{v\in T_x\M: \text{there exists $(x',v')\in \supp(f_0)$ such that $\phi_t(x',v')=(x,v)$} \Big\}=\emptyset,$$ and therefore $\rho(f)(t,x)=0$.
\end{proof}

\begin{remark}
Using the previous lemma, we can obtain a \emph{non-optimal} decay estimate for the spatial density appealing to a geometric argument based on Rauch's comparison theorem and the hyperbolic law of cosines. See Appendix \ref{app_subsection_decay_spatial_density_AH} for a short proof of a non-optimal decay estimate for Vlasov fields on non-trapping asymptotically hyperbolic Riemannian manifolds. 
\end{remark}

Finally, we proceed to prove Theorem \ref{thm_decay_spatial_density_ah_manif}. First we obtain the estimate \eqref{estimate_spatial_density_no_derivatives_ah} and then the estimate \eqref{estimate_spatial_density_with_derivatives_ah}. 

\begin{proof}[Proof of the estimate (\ref{estimate_spatial_density_no_derivatives_ah})]

Let us first introduce some notation. Given $x\in \M$ and $t\geq 0$ we define $$\Omega(t,x):=\Big\{v\in T_x\M:f(t,x,v)\ne 0\Big\}.$$ Observe that the set $\Omega(t,x)$ can be considered as the domain of integration of the spatial density. Since $f(t,x,v)=f_0(\phi_{-t}(x,v))$, we have that
\begin{equation}\label{Omega_5}
    \Omega(t,x)=\Big\{v\in T_x\M: \text{there exists $(x',v')\in \supp(f_0)$ such that $\phi_t(x',v')=(x,v)$} \Big\},
\end{equation}
and that $\|f(t)\|_{L^{\infty}_{x,v}}=\|f(0)\|_{L^{\infty}_{x,v}}$.

We proceed to estimate the spatial density $\rho(f)(t,x)$. Fix $x=(r,\theta)\in\M$ and $t\ge 0$. It follows from \eqref{estimate_velocity_coordinate_ah_space} and \eqref{Omega_5} that if  $(v^\theta,v^r)\in \Omega(t,x)$, then 
\begin{equation}\label{equation_decay_ptheta_ah}
      |\sinh rv^{\theta}|\leq \dfrac{C}{\exp( \alpha t)},\qquad  \bigg(E^2- \dfrac{C}{\exp{(2\alpha t)}}\bigg)^{\frac{1}{2}}\leq |v^r|\le E,
\end{equation}
and therefore $$\vol_{T_x\M}(\Omega(t,x))=  \int_{\Omega(t,x)}\dvol_{T_x\M}(v) \le \frac{D_0}{\exp(\alpha t)},$$ where $D_0$ is a positive constant depending on the support of $f_0$. Finally, we have
$$ |\rho(f)(t,x)|=\Big|\int f(t,x,v)\dvol_{T_x\M}(v)\Big|\leq \|f_0\|_{L^{\infty}_{x,v}} 
    \ \int_{\Omega(t,x)}\dvol_{T_x\M}(v) 
    \leq   \dfrac{D_0}{\exp(\alpha t)}\|f_0\|_{L^{\infty}_{x,v}}.$$ 
    \end{proof}

Next, we prove decay estimates for derivatives of the spatial density induced by Vlasov fields supported on $\D_{\alpha}$. First, we consider the orthogonal parallel frame $\{\dot{\gamma}, N\}$ along an arbitrary geodesic $\gamma$. We write this frame in local coordinates by
\begin{equation}\label{frame_parallel_orthogonal_dimension_two_AH}
\dot{\gamma}:=v^r\partial_{r}+\sinh r v^{\theta}\dfrac{\partial_{\theta}}{\sinh r}, \qquad N:=-N^r\partial_{r}+N^{\theta}\dfrac{\partial_{\theta}}{\sinh r}, 
\end{equation}
where $N$ defined as the unique vector field such that $$g(N,N)=g(\dot{\gamma},\dot{\gamma}),\qquad g(N,\dot{\gamma})=0,$$ and $\{\dot{\gamma},N\}$ is a positively oriented basis of $T_x\M$. Moreover, we have the estimates
\begin{alignat}{3}
N^r&= \sinh r v^{\theta}+O(e^{-\beta r}), \qquad &&\partial_{v^r}N^r=O(e^{-\beta r}),\qquad && \partial_{v^{\theta}}N^r= \sinh r +O(e^{-\beta r}),\label{estimates_Nr}\\ 
N^{\theta}&= v^{r}+O(e^{-\beta r}),\qquad &&\partial_{v^r}N^{\theta}= 1+O(e^{- \beta r}),\qquad && \partial_{v^{\theta}}N^{\theta}= O(e^{-\beta r}),\label{estimates_Ntheta}
\end{alignat}
since $(\M,g)$ is asymptotically hyperbolic. We also observe that the frame $\{\partial_r,(\sinh r)^{-1} \partial_{\theta}\}$ can be written in terms of the parallel frame $\{\dot{\gamma},N\}$ as 
$$ \partial_r =\frac{N^{\theta}\dot{\gamma}-\sinh rv^{\theta}N}{N^\theta v^{r}+\sinh rv^{\theta}N^{r}},\qquad \dfrac{\partial_{\theta}}{\sinh r}=\frac{N^r\dot{\gamma}+v^rN}{N^r\sinh r v^{\theta}+v^rN^{\theta}}.$$

Before proving the estimates in \eqref{estimate_spatial_density_with_derivatives_ah}, we obtain a technical lemma concerning the decay of the solutions $q_u$ and $q_s$ to the Riccati equation \eqref{Riccati_eqn_invariant_geo}. The estimates in the following lemma will be required to bound the error terms coming out when integrating by parts the vertical part of the unstable vector field $U$. 

\begin{lemma}\label{lemma_estimates_riccati_eqn}
The solution $q_u:C^{2-}(T\M)\to (0,\infty)$ to the Riccati equation \eqref{Riccati_eqn_invariant_geo} satisfies
\begin{alignat}{3}
\Big|\frac{q_u}{E}-1\Big|=O(e^{-\beta r}),\qquad \Big|\partial_{v^r}\Big(\frac{q_u}{E}\Big)\Big|=O(e^{-\beta r}),\qquad \Big|\frac{1}{\sinh r}\partial_{v^{\theta}}\Big(\frac{q_u}{E}\Big)\Big|=O(e^{-\beta r}).
\end{alignat}
\end{lemma}

\begin{proof}
By Hurder--Katok \cite{HK90} the function $q_u$ belongs to $C^{2-}(T\M)$. We recall that the function $q_u$ parametrizes the unstable invariant distribution $E_u$ of the geodesic flow in $(\M,g)$. See Section \ref{subsec_jacobi_neg} for more details. The distribution $E_u$ is also the tangent of the unstable manifolds of the geodesic flow. By the standard stable manifold theorem \cite[Chapter 17]{KH} applied to the geodesic flow in $(\M,g)$, the function $q_u$ appears when proving the contraction property of the action of the flow map $\phi_t$ on unstable graphs along the flow. Since $(\M,g)$ is asymptotically hyperbolic, the generator of the geodesic flow $X$ converges like $e^{-\beta r}$ to the generator $X_{\H^2}$ of the geodesic flow in hyperbolic space. Thus, the unstable manifolds associated to the geodesic flow in $(\M,g)$ converge like $e^{-\beta r}$ to the unstable manifolds associated to the geodesic flow in $(\H^2,g_{\H^2})$. In particular, we obtain the estimates stated above. 
\end{proof}

The proof of Lemma \ref{lemma_estimates_riccati_eqn} follows the same strategy applied by Hintz \cite{H21} to show a stable manifold theorem for a $C^k$ normally hyperbolic flow converging to a $C^k$ stationary flow. Specifically, \cite{H21} shows that the stable and unstable manifolds of a perturbation of the normally hyperbolic flow converge in $C^k$ to their stationary counterparts.

Finally, we proceed to prove the pointwise decay estimates for Vlasov fields supported on $\D_{\alpha}$.

\begin{proof}[Proof of the estimate (\ref{estimate_spatial_density_with_derivatives_ah})]

Since every spatial derivative of the spatial density is equal to the spatial density of the corresponding horizontal derivative of the distribution function, we have 
\begin{align*}
    \partial_{r}\rho(f)(t,x)&= \rho(\Hor_{(x,v)}(\partial_{r})f)(t,x)=:\rho(H_r f)(t,x),\\
    \dfrac{1}{\sinh r}\partial_{\theta}\rho(f)(t,x)&= \rho\Big(\Hor_{(x,v)}\Big(\dfrac{\partial_{\theta}}{\sinh r}\Big)f\Big)(t,x)=: \rho(H_{\theta}f)(t,x),
\end{align*}
where the horizontal vector fields $H_r$ and $H_{\theta}$ can be written as $$H_r =\frac{N^{\theta}X-\sinh rv^{\theta}H}{N^\theta v^{r}+\sinh rv^{\theta}N^{r}},\qquad H_{\theta}=\frac{N^rX+v^rH}{N^r\sinh r v^{\theta}+v^rN^{\theta}},$$ in terms of the vector fields $H=\Hor_{(x,v)}(N)$ and $X=\Hor_{(x,v)}(v)$. 

\textbf{Estimate for the derivative $(\sinh r)^{-1}\partial_{\theta}\rho(f)$.} By the previous decomposition of $H_{\theta}$, the angular derivative $(\sinh r)^{-1}\partial_{\theta}\rho(f)$ can be decomposed into
$$\dfrac{1}{\sinh r}\partial_{\theta}\rho(f)(t,x)= \rho\Big(\frac{N^rXf}{N^r\sinh r v^{\theta}+v^rN^{\theta}}\Big)+\rho\Big(\frac{v^rHf}{N^r\sinh r v^{\theta}+v^rN^{\theta}}\Big)=:A_1+B_1.$$

On the one hand, the first term $A_1$ can be written using the commuting vector field $tX+Y$ by $$A_1=\dfrac{1}{t}\int \frac{N^r(tX+Y)f}{N^r\sinh r v^{\theta}+v^rN^{\theta}} \dvol_{T_x\M}(v)-\dfrac{1}{t}\int \frac{N^rYf}{N^r\sinh r v^{\theta}+v^rN^{\theta}} \dvol_{T_x\M}(v).$$
The second term in the RHS can be further decomposed as
\begin{align*}
    \int \frac{N^r}{N^r\sinh r v^{\theta}+v^rN^{\theta}}Yf \dvol_{T_x\M}(v)&=\int \frac{N^r v^r}{N^r\sinh r v^{\theta}+v^rN^{\theta}}\partial_{v^r}f \dvol_{T_x\M}(v)\\
    &\qquad +\int \frac{N^r v^{\theta}}{N^r\sinh r v^{\theta}+v^rN^{\theta}}\partial_{v^{\theta}}f \dvol_{T_x\M}(v)\\
    &=:C_1+C_2.
\end{align*}

Integrating by parts the first term $C_1$, we have
\begin{align*}
C_1&=\int \frac{N^r v^r}{N^r\sinh r v^{\theta}+v^rN^{\theta}}\partial_{v^r}f \dvol_{T_x\M}(v)\\
&=-\int \frac{\partial_{v^r}N^r v^r+N^r}{N^r\sinh r v^{\theta}+v^rN^{\theta}} f\dvol_{T_x\M}(v) \\
&\qquad \qquad+\int \frac{N^r v^r(\partial_{v^r}N^r\sinh r v^{\theta}+v^r\partial_{v^r}N^{\theta}+N^{\theta})}{(N^r\sinh r v^{\theta}+v^rN^{\theta})^2} f\dvol_{T_x\M}(v).
\end{align*}
We make use of the decay of the velocity support of the distribution together with \eqref{equation_decay_ptheta_ah}, \eqref{estimates_Nr}, and \eqref{estimates_Ntheta}, to obtain 
\begin{align*}
|C_1|&\lesssim \int \Big|\frac{N^r}{N^r\sinh r v^{\theta}+v^rN^{\theta}}\Big| f+ \Big|\frac{N^r v^r N^{\theta}}{ (N^r\sinh r v^{\theta}+v^rN^{\theta})^2}\Big| f\dvol_{T_x\M}(v)+e^{-2\alpha t}\|f_{0}\|_{L^{\infty}_{x,v}} \\
&\lesssim e^{-2\alpha t}\|f_{0}\|_{L^{\infty}_{x,v}}.
\end{align*}

Integrating by parts the second term $C_2$, we have
\begin{align*}
C_2&=\int \frac{N^r v^{\theta}}{N^r\sinh r v^{\theta}+v^rN^{\theta}}\partial_{v^{\theta}}f \dvol_{T_x\M}(v)\\
&=-\int \frac{\partial_{v^{\theta}}N^r v^{\theta}+N^r}{N^r\sinh r v^{\theta}+v^rN^{\theta}}f \dvol_{T_x\M}(v)\\
&\qquad\qquad +\int \frac{N^r v^{\theta}(\partial_{v^{\theta}}N^r\sinh r v^{\theta}+N^r\sinh r+v^r\partial_{v^{\theta}}N^{\theta})}{(N^r\sinh r v^{\theta}+v^rN^{\theta})^2}f \dvol_{T_x\M}(v).
\end{align*}
We make use of the decay of the velocity support of the distribution together with \eqref{equation_decay_ptheta_ah}, \eqref{estimates_Nr}, and \eqref{estimates_Ntheta}, to obtain 
\begin{align*}
|C_2|&\lesssim \int \Big|\frac{N^r}{N^r\sinh r v^{\theta}+v^rN^{\theta}}\Big|f \dvol_{T_x\M}(v)\\
&\qquad\qquad +\int \frac{(N^r)^2 |v^{\theta}\sinh r|}{(N^r\sinh r v^{\theta}+v^rN^{\theta})^2}f \dvol_{T_x\M}(v)+e^{-2\alpha t}\|f_{0}\|_{L^{\infty}_{x,v}}\\
&\lesssim e^{-2\alpha t}\|f_{0}\|_{L^{\infty}_{x,v}}.
\end{align*} 

Using the same arguments for the first term of $A_1$ and applying the estimates derived for $C_1$ and $C_2$, we obtain
\begin{align*}
    |A_1|&\leq 
\Big|\dfrac{1}{t}\int \frac{N^r}{N^r\sinh r v^{\theta}+v^rN^{\theta}}(tX+Y)f \dvol_{T_x\M}(v)\Big|+\dfrac{1}{t}|C_1|+\dfrac{1}{t}|C_2|\\
&\lesssim \dfrac{1}{te^{2\alpha t}}\|Yf_0\|_{L^{\infty}_{x,v}} + \dfrac{1}{te^{2\alpha t}}\|f_0\|_{L^{\infty}_{x,v}}\\
&\lesssim \dfrac{1}{te^{2\alpha t}}\|f_0\|_{W^{1,\infty}_{x,v}}.
\end{align*}

On the other hand, the second term $B_1$ can be written using the commuting unstable vector field $U$ by
\begin{align*}
    B_1&=\int \frac{v^r(H+q_uV)f}{N^r\sinh r v^{\theta}+v^rN^{\theta}} \dvol_{T_x\M}(v)-\int \frac{v^rq_uVf}{N^r\sinh r v^{\theta}+v^rN^{\theta}} \dvol_{T_x\M}(v),\\
    &=\int \dfrac{1}{e^{\int_0^t q_u d\tau}}\frac{v^rUf}{N^r\sinh r v^{\theta}+v^rN^{\theta}}\dvol_{T_x\M}(v)-\int\frac{v^rq_uVf}{N^r\sinh r v^{\theta}+v^rN^{\theta}} \dvol_{T_x\M}(v),
\end{align*}
where the second term in the RHS can be integrated by parts
\begin{align*}
    \int\frac{v^rq_uVf}{N^r\sinh r v^{\theta}+v^rN^{\theta}} \dvol_{T_x\M}(v)&=\int \frac{v^r N^{\theta}}{N^r\sinh r v^{\theta}+v^rN^{\theta}}q_u\dfrac{\partial_{v^{\theta}}f}{\sinh r} \dvol_{T_x\M}(v)\\
    &\qquad -\int \frac{v^r N^{r}}{N^r\sinh r v^{\theta}+v^rN^{\theta}} q_u\partial_{v^r}f \dvol_{T_x\M}(v)\\
    &=:D_1-D_2.
\end{align*}

Integrating by parts the first term $D_1$, we have
\begin{align*}
D_1&=\int \frac{v^r N^{\theta}E}{N^r\sinh r v^{\theta}+v^rN^{\theta}}\frac{q_u}{E}\dfrac{\partial_{v^{\theta}}f}{\sinh r} \dvol_{T_x\M}(v)\\
&=-\int \frac{\partial_{v^{\theta}}N^{\theta} v^r}{N^r\sinh r v^{\theta}+v^rN^{\theta}} \dfrac{q_u f}{\sinh r}+ \frac{v^r N^{\theta}\partial_{v^{\theta}}E}{N^r\sinh r v^{\theta}+v^rN^{\theta}}\frac{q_u}{E}\dfrac{\partial_{v^{\theta}}f}{\sinh r} \dvol_{T_x\M}(v)\\
&\qquad -\int \frac{v^r N^{\theta}E}{N^r\sinh r v^{\theta}+v^rN^{\theta}}\dfrac{1}{\sinh r}\partial_{v^{\theta}}\Big(\frac{q_u}{E}\Big)f \dvol_{T_x\M}(v)\\
&\qquad \qquad+\int \frac{N^{\theta} v^r(\partial_{v^{\theta}}N^r\sinh r v^{\theta}+N^r\sinh r+v^r\partial_{v^{\theta}}N^{\theta})}{(N^r\sinh r v^{\theta}+v^rN^{\theta})^2} \dfrac{q_u}{\sinh r}f\dvol_{T_x\M}(v).
\end{align*}
We make use of the decay of the velocity support of the distribution together with \eqref{equation_decay_ptheta_ah}, \eqref{estimates_Nr}, and \eqref{estimates_Ntheta}, to obtain 
\begin{align*}
|D_1|&\lesssim \int \Big|\frac{v^r N^{\theta} E}{N^r\sinh r v^{\theta}+v^rN^{\theta}}\dfrac{1}{\sinh r}\partial_{v^{\theta}}\Big(\frac{q_u}{E}\Big)\Big|f \dvol_{T_x\M}(v)\\
&\qquad +\int \frac{|N^{\theta}  v^r N^r q_u|}{(N^r\sinh r v^{\theta}+v^rN^{\theta})^2} f\dvol_{T_x\M}(v)+e^{-2\alpha t}\|f_{0}\|_{L^{\infty}_{x,v}} \\
&\lesssim e^{-2\alpha t}\|f_{0}\|_{L^{\infty}_{x,v}}.
\end{align*}

Integrating by parts the second term $D_2$, we have
\begin{align*}
D_2&=\int \frac{v^r N^{r}E}{N^r\sinh r v^{\theta}+v^rN^{\theta}} \frac{q_u}{E}\partial_{v^r}f \dvol_{T_x\M}(v)\\
&=-\int \frac{\partial_{v^r}N^{r} v^r+N^r}{N^r\sinh r v^{\theta}+v^rN^{\theta}} q_uf+\frac{v^r N^{r}E}{N^r\sinh r v^{\theta}+v^rN^{\theta}}\partial_{v^r}\Big(\frac{q_u}{E}\Big)f \dvol_{T_x\M}(v)\\
&\qquad+\int \frac{N^{r} v^r(\partial_{v^r}N^r\sinh r v^{\theta}+N^{\theta}+v^r\partial_{v^r}N^{\theta})}{(N^r\sinh r v^{\theta}+v^rN^{\theta})^2} q_uf\dvol_{T_x\M}(v)\\
&\qquad -\int \frac{v^r N^{r}\partial_{v^r}E}{N^r\sinh r v^{\theta}+v^rN^{\theta}} \frac{q_u}{E}f \dvol_{T_x\M}(v).
\end{align*}
We make use of the decay of the velocity support of the distribution together with \eqref{equation_decay_ptheta_ah}, \eqref{estimates_Nr}, and \eqref{estimates_Ntheta}, to obtain 
\begin{align*}
|D_2|&\lesssim \int \Big|\frac{N^r}{N^r\sinh r v^{\theta}+v^rN^{\theta}} q_u\Big|f + \Big|\frac{v^r N^{r}E}{N^r\sinh r v^{\theta}+v^rN^{\theta}}\partial_{v^r}\Big(\frac{q_u}{E}\Big)\Big|f \dvol_{T_x\M}(v) \\
&\qquad +\int \Big|\frac{N^{r} v^r N^{\theta}q_u}{(N^r\sinh r v^{\theta}+v^rN^{\theta})^2} \Big| f\dvol_{T_x\M}(v)+e^{-2\alpha t}\|f_{0}\|_{L^{\infty}_{x,v}} \\
&\lesssim e^{-2\alpha t}\|f_{0}\|_{L^{\infty}_{x,v}}.
\end{align*}

We obtain that 
$$B_1=\int \dfrac{1}{e^{\int_0^t q_u d\tau}}\frac{v^r}{N^r\sinh r v^{\theta}+v^rN^{\theta}}Uf\dvol_{T_x\M}(v)-D_1+D_2.$$ Using the same arguments for the first term of $B_1$ and applying the estimates derived for $D_1$ and $D_2$, we obtain
\begin{align*}
    |B_1|&\leq 
\Big|\int \dfrac{1}{e^{\int_0^t q_u d\tau}}\frac{v^r}{N^r\sinh r v^{\theta}+v^rN^{\theta}}Uf\dvol_{T_x\M}(v)\Big|+|D_1|+|D_2|\\
& \lesssim \dfrac{1}{e^{2\alpha t}}\|Uf_0\|_{L^{\infty}_{x,v}} +\frac{1}{e^{2\alpha t}}\|f_0\|_{L^{\infty}_{x,v}}\\
&\lesssim \dfrac{1}{e^{2\alpha t}}\|f_0\|_{W^{1,\infty}_{x,v}}.
\end{align*}
Therefore, we obtain that the angular derivative of the spatial density $(\sinh r)^{-1}\partial_{\theta}\rho(f)$ is bounded above by
$$|(\sinh r)^{-1}\partial_{\theta}\rho(f)|\leq |A_1|+|B_1|\lesssim \|f_0\|_{W^{1,\infty}_{x,v}}e^{-{2\alpha t}}.$$

\textbf{Estimate for the derivative $\partial_{r}\rho(f)$.} By the previous decomposition of $H_r$, the radial derivative $\partial_{r}\rho(f)$ can be decomposed into
$$\partial_{r}\rho(f)(t,x)=\rho\Big( \frac{N^{\theta}}{N^\theta v^{r}+\sinh rv^{\theta}N^{r}}Xf\Big)-\rho\Big(\frac{\sinh rv^{\theta}}{N^\theta v^{r}+\sinh rv^{\theta}N^{r}}Hf\Big)=:A_2-B_2.$$

First, we write $A_2$ using the commuting vector field $tX+Y$ by 
\begin{align*}
A_2&=\rho\Big( \frac{N^{\theta}}{N^\theta v^{r}+\sinh rv^{\theta}N^{r}}Xf\Big)\\
&=\dfrac{1}{t}\int \frac{N^{\theta}(tX+Y)f}{N^\theta v^{r}+\sinh rv^{\theta}N^{r}} \dvol_{T_x\M}(v)-\dfrac{1}{t}\int \frac{N^{\theta}Yf}{N^\theta v^{r}+\sinh rv^{\theta}N^{r}} \dvol_{T_x\M}(v),
\end{align*}
where the second term in the RHS can be decomposed as
\begin{align*}
    \int \frac{N^{\theta}}{N^\theta v^{r}+\sinh rv^{\theta}N^{r}}Yf \dvol_{T_x\M}(v)&=\int \frac{N^{\theta}v^r}{N^\theta v^{r}+\sinh rv^{\theta}N^{r}}\partial_{v^r}f \dvol_{T_x\M}(v)\\
    &\quad+\int \frac{N^{\theta}v^{\theta}}{N^\theta v^{r}+\sinh rv^{\theta}N^{r}}\partial_{v^{\theta}}f \dvol_{T_x\M}(v)\\
    &=E_1+E_2.
\end{align*}

Integrating by parts the first term, we have
\begin{align*}
E_1&=\int \frac{N^{\theta}v^r}{N^\theta v^{r}+\sinh rv^{\theta}N^{r}}\partial_{v^r}f \dvol_{T_x\M}(v)\\
&=-\int \frac{\partial_{v^r}N^{\theta}v^r+N^{\theta}}{N^\theta v^{r}+\sinh rv^{\theta}N^{r}}f \dvol_{T_x\M}(v)\\
&\qquad+\int \frac{N^{\theta}v^r(\partial_{v^r}N^\theta v^{r}+N^\theta+\sinh rv^{\theta}\partial_{v^r}N^{r})}{(N^\theta v^{r}+\sinh rv^{\theta}N^{r})^2}f \dvol_{T_x\M}(v).
\end{align*}
We make use of the decay of the velocity support of the distribution together with \eqref{equation_decay_ptheta_ah}, \eqref{estimates_Nr}, and \eqref{estimates_Ntheta}, to obtain 
\begin{align*}
|E_1|&\lesssim \int \Big|\frac{N^{\theta}}{N^\theta v^{r}+\sinh rv^{\theta}N^{r}}\Big|f + \frac{(N^{\theta})^2|v^r|}{(N^\theta v^{r}+\sinh rv^{\theta}N^{r})^2} f\dvol_{T_x\M}(v)+e^{-\alpha t}\|f_{0}\|_{L^{\infty}_{x,v}} \\
&\lesssim e^{-\alpha t}\|f_{0}\|_{L^{\infty}_{x,v}}.
\end{align*}
Integrating by parts the second term, we have
\begin{align*}
E_2&=\int \frac{N^{\theta}v^{\theta}}{N^\theta v^{r}+\sinh rv^{\theta}N^{r}}\partial_{v^{\theta}}f \dvol_{T_x\M}(v)\\
&=-\int \frac{\partial_{v^{\theta}}N^{\theta}v^{\theta}+N^{\theta}}{N^\theta v^{r}+\sinh rv^{\theta}N^{r}}f \dvol_{T_x\M}(v)\\
&\qquad +\int \frac{N^{\theta}v^{\theta}(\partial_{v^{\theta}}N^\theta v^{r}+\sinh rv^{\theta}\partial_{v^{\theta}}N^{r}+\sinh rN^{r})}{(N^\theta v^{r}+\sinh rv^{\theta}N^{r})^2}f \dvol_{T_x\M}(v).
\end{align*}
We make use of the decay of the velocity support of the distribution together with \eqref{equation_decay_ptheta_ah}, \eqref{estimates_Nr}, and \eqref{estimates_Ntheta}, to obtain 
\begin{align*}
|E_2|&\lesssim \int \Big|\frac{N^{\theta}}{N^\theta v^{r}+\sinh rv^{\theta}N^{r}}\Big|f \dvol_{T_x\M}(v)\\
&\qquad +\int \frac{|N^{\theta}N^{r}v^{\theta}\sinh r|}{(N^\theta v^{r}+\sinh rv^{\theta}N^{r})^2} f\dvol_{T_x\M}(v)+e^{-2\alpha t}\|f_{0}\|_{L^{\infty}_{x,v}} \\
&\lesssim e^{-\alpha t}\|f_{0}\|_{L^{\infty}_{x,v}}.
\end{align*}

Using the same arguments for the first term of $A_2$ and applying the estimates derived for $E_1$ and $E_2$, we obtain
\begin{align*}
    |A_2|&\leq 
\Big|\dfrac{1}{t}\int \frac{N^{\theta}}{N^\theta v^{r}+\sinh rv^{\theta}N^{r}}(tX+Y)f \dvol_{T_x\M}(v)\Big|+\dfrac{1}{t}|E_1|+\dfrac{1}{t}|E_2|\\
&\lesssim \dfrac{1}{te^{\alpha t}}\|Yf_0\|_{L^{\infty}_{x,v}} + \dfrac{1}{te^{\alpha t}}\|f_0\|_{L^{\infty}_{x,v}}\\
&\lesssim \dfrac{1}{te^{\alpha t}}\|f_0\|_{W^{1,\infty}_{x,v}}.
\end{align*}

Similarly, the second term $B_2$ can be written using the unstable vector field $U$ by
\begin{align*}
    B_2&=\rho\Big(\frac{\sinh rv^{\theta}}{N^\theta v^{r}+\sinh rv^{\theta}N^{r}}Hf\Big)(t,x)\\
    &=\int \frac{\sinh rv^{\theta}(H+q_uV)f}{N^\theta v^{r}+\sinh rv^{\theta}N^{r}}\dvol_{T_x\M}(v)-\int \frac{\sinh rv^{\theta}q_uVf}{N^\theta v^{r}+\sinh rv^{\theta}N^{r}} \dvol_{T_x\M}(v)\\
    &=\int \dfrac{1}{e^{\int_0^t q_u d\tau}}\frac{\sinh rv^{\theta}Uf}{N^\theta v^{r}+\sinh rv^{\theta}N^{r}}\dvol_{T_x\M}(v)-\int \frac{\sinh rv^{\theta}q_uVf}{N^\theta v^{r}+\sinh rv^{\theta}N^{r}} \dvol_{T_x\M}(v),
\end{align*}
where the second term in the RHS can be integrated by parts
\begin{align*}
\int \frac{\sinh rv^{\theta}q_uVf}{N^\theta v^{r}+\sinh rv^{\theta}N^{r}} \dvol_{T_x\M}(v)&=\int \frac{\sinh rv^{\theta}N^{\theta}}{N^\theta v^{r}+\sinh rv^{\theta}N^{r}}q_u\dfrac{\partial_{v^{\theta}}f}{\sinh r} \dvol_{T_x\M}(v)\\
    &\qquad -\int \frac{\sinh rv^{\theta}N^r}{N^\theta v^{r}+\sinh rv^{\theta}N^{r}} q_u\partial_{v^r}f \dvol_{T_x\M}(v)\\
    &=:F_1-F_2.
\end{align*}

Integrating by parts the first term $F_1$, we have
\begin{align*}
F_1&=\int \frac{\sinh rv^{\theta}N^{\theta}E}{N^\theta v^{r}+\sinh rv^{\theta}N^{r}}\frac{q_u}{E}\dfrac{\partial_{v^{\theta}}f}{\sinh r} \dvol_{T_x\M}(v)\\
&=-\int \frac{\sinh rv^{\theta}\partial_{v^{\theta}}N^{\theta}+\sinh rN^{\theta}}{N^r\sinh r v^{\theta}+v^rN^{\theta}} \dfrac{q_uf}{\sinh r}+ \frac{\sinh rv^{\theta}N^{\theta}\partial_{v^{\theta}}E}{N^\theta v^{r}+\sinh rv^{\theta}N^{r}}\frac{q_u f}{E \sinh r}\dvol_{T_x\M}(v) \\
&\qquad -\int \frac{\sinh rv^{\theta}N^{\theta}E}{N^r\sinh r v^{\theta}+v^rN^{\theta}}\dfrac{\partial_{v^{\theta}}}{\sinh r}\Big(\frac{q_u}{E}\Big)f \dvol_{T_x\M}(v)\\
&\qquad \qquad+\int \frac{\sinh rv^{\theta}N^{\theta}(\partial_{v^{\theta}}N^r\sinh r v^{\theta}+N^r\sinh r+v^r\partial_{v^{\theta}}N^{\theta})}{(N^r\sinh r v^{\theta}+v^rN^{\theta})^2} \dfrac{q_u}{\sinh r}f\dvol_{T_x\M}(v).
\end{align*}
We make use of the decay of the velocity support of the distribution together with \eqref{equation_decay_ptheta_ah}, \eqref{estimates_Nr}, and \eqref{estimates_Ntheta}, to obtain 
\begin{align*}
|F_1|&\lesssim \int \Big|\frac{N^{\theta}q_u}{N^r\sinh r v^{\theta}+v^rN^{\theta}} \Big|f +\Big|\frac{\sinh r v^{\theta}N^{\theta}E}{N^r\sinh r v^{\theta}+v^rN^{\theta}}\frac{\partial_{v^{\theta}}}{\sinh r}\Big(\frac{q_u}{E}\Big)\Big|f \dvol_{T_x\M}(v)\\
&\qquad +\int \frac{|\sinh rv^{\theta}N^{\theta}N^r|}{(N^r\sinh r v^{\theta}+v^rN^{\theta})^2} q_u f\dvol_{T_x\M}(v)+e^{-\alpha t}\|f_{0}\|_{L^{\infty}_{x,v}} \\
&\lesssim e^{-\alpha t}\|f_{0}\|_{L^{\infty}_{x,v}}.
\end{align*}

Integrating by parts the second term $F_2$, we have
\begin{align*}
F_2&=\int \frac{\sinh rv^{\theta} N^{r}E}{N^r\sinh r v^{\theta}+v^rN^{\theta}} \frac{q_u}{E}\partial_{v^r}f \dvol_{T_x\M}(v)\\
&=-\int \frac{\sinh rv^{\theta}\partial_{v^r}N^{r}}{N^r\sinh r v^{\theta}+v^rN^{\theta}} q_uf+ \frac{\sinh rv^{\theta} N^{r}E}{N^r\sinh r v^{\theta}+v^rN^{\theta}}\partial_{v^r}\Big(\frac{q_u}{E}\Big)f \dvol_{T_x\M}(v)\\
&\qquad \qquad+\int \frac{N^{r} \sinh rv^{\theta} (\partial_{v^r}N^r\sinh r v^{\theta}+N^{\theta}+v^r\partial_{v^r}N^{\theta})}{(N^r\sinh r v^{\theta}+v^rN^{\theta})^2} q_uf\dvol_{T_x\M}(v)\\
&\qquad \qquad -\int \frac{\sinh rv^{\theta} N^{r}\partial_{v^r}E}{N^r\sinh r v^{\theta}+v^rN^{\theta}} \frac{q_u}{E}\partial_{v^r}f \dvol_{T_x\M}(v).
\end{align*}
We make use of the decay of the velocity support of the distribution together with \eqref{equation_decay_ptheta_ah}, \eqref{estimates_Nr}, and \eqref{estimates_Ntheta}, to obtain 
\begin{align*}
|F_2|&\lesssim \int \Big|\frac{\sinh rv^{\theta} N^{r}E}{N^r\sinh r v^{\theta}+v^rN^{\theta}}\partial_{v^r}\Big(\frac{q_u}{E}\Big)\Big|f + \frac{|N^{\theta}N^{r} \sinh rv^{\theta}q_u| }{(N^r\sinh r v^{\theta}+v^rN^{\theta})^2} f\dvol_{T_x\M}(v)+\frac{\|f_{0}\|_{L^{\infty}_{x,v}}}{e^{\alpha t}} \\
&\lesssim e^{-\alpha t}\|f_{0}\|_{L^{\infty}_{x,v}}.
\end{align*}

Using the same arguments for the first term of $B_2$ and applying the estimates derived for $F_1$ and $F_2$, we obtain
\begin{align*}
    |B_2|&\leq 
\Big|\int \dfrac{1}{e^{\int_0^t q_u d\tau}}\frac{\sinh rv^{\theta}}{N^\theta v^{r}+\sinh rv^{\theta}N^{r}}Uf \dvol_{T_x\M}(v)\Big|+|F_1|+|F_2|\\
&\lesssim \dfrac{1}{e^{3\alpha t}}\|Uf_0\|_{L^{\infty}_{x,v}} + \dfrac{1}{e^{\alpha t}}\|f_0\|_{L^{\infty}_{x,v}}\\
&\lesssim \dfrac{1}{e^{\alpha t}}\|f_0\|_{W^{1,\infty}_{x,v}}.
\end{align*}
Finally, we obtain that the radial derivative of the spatial density $\partial_{r}\rho(f)$ is bounded above by
$$|\partial_{r}\rho(f)|\leq |A_2|+|B_2|\lesssim \|f_0\|_{W^{1,\infty}_{x,v}}e^{-{\alpha t}}.$$ 

\end{proof}

\addtocontents{toc}{\protect\setcounter{tocdepth}{0}}
\appendix

\section{Proof of (non-optimal) decay for Vlasov fields supported on \texorpdfstring{$\D_{\alpha}$}{DA}}\label{app_subsection_decay_spatial_density_AH}
\addtocontents{toc}{\protect\setcounter{tocdepth}{1}}

Finally, we obtain a non-optimal decay estimate for Vlasov fields on a non-trapping asymptotically hyperbolic manifold supported on $\mathcal{D}_{\alpha}$. The following proposition follows by using Rauch comparison theorem and the hyperbolic law of cosines.

\begin{proposition}
\label{proposition_decay_AH_general_rauch}
Let $\alpha>0$. Let $(\M,g)$ be asymptotically hyperbolic  and non-trapping. Let $f_0$ be an initial data for the Vlasov equation on $(\M,g)$ that is compactly supported on $\D_{\alpha}$. Then, for every $\epsilon>0$, there exists $C_{\e}\geq 0$, such that the spatial density induced by the corresponding Vlasov field $f$ satisfies
$$|\rho(f)(t,x)|\leq \frac{C_{\e}}{\exp{(\alpha(n-1)(1-\e)t)}}\|f_0\|_{L^{\infty}_{x,v}},$$
for every $t\geq 0$ and every $x\in \M$.
\end{proposition}

\begin{proof} Fix $\delta>0$ and choose $r$ sufficiently large such that the sectional curvatures of $\Phi^*g$ in $\H^n\setminus \overline{B(r)}$ lies in $[-1-\delta,-1+\delta]$, or equivalently, that the same holds for the sectional curvatures of $g$ in $\M\setminus K_r$. We assume that $\supp (f_0)\subset \widehat{K_r}$ and we set $D=\diam (\widehat{K_r})$. It suffices to prove the estimate when $x$ is at distance at least $3D$ from $K_r$ (see Lemma \ref{lem:ahyp}). In order to bound $\rho(f)(t,x)$ we will estimate $\vol_{T_x\M} \Omega(x,t)$. Firstly, note that if $v\in\Omega(t,x)$, then $\phi_t(x,-v)\in \widehat{K}_r,$ and $|v|_g\ge \alpha$. In particular, if $v_1,v_2\in \Omega(t,x),$ then $d(\phi_{t}(x,-v_1),\phi_{t}(x,-v_2))\leq D.$ By the triangle inequality we obtain 
\begin{equation}\label{1ahyp}
d(\phi_{t-{2D}/{|v_1|_g}}(x,-v_1),\phi_{t-{2D}/{|v_2|_g}}(x,-v_2))\leq 5D.
\end{equation} 
Note that the geodesic triangle with vertices $x$, $\phi_{t-{2D}/{|v_1|_g}}(x,-v_1)$, and  $\phi_{t-{2D}/{|v_2|_g}}(x,-v_2)$ is contained in $\M\setminus K_r$, where the sectional curvature is bounded between $-1-\delta$ and $-1+\delta$. It  follows from inequality \eqref{1ahyp}, the Rauch comparison theorem and the hyperbolic law of cosines that the angle between $v_1$ and $v_2$ is smaller than $Ce^{-\alpha\sqrt{1-\delta}t}$, for some constant $C$ that depends on the support of $f_0$ and $\delta$, but it is independent of $t$ and $x$. We conclude that $\vol_{T_x\M} \Omega(x,t)\le C'e^{-\alpha(n-1)\sqrt{1-\delta}t}$, and therefore 
$$|\rho(f)(t,x)|\le \|f_0\|_{L^{\infty}_{x,v}}\vol_{T_x\M} \Omega(x,t)\le C'\|f_0\|_{L^{\infty}_{x,v}}e^{-\alpha(n-1)\sqrt{1-\delta}t},$$
for some constant $C'=C'(\delta,f_0)$ independent of $t$ and $x$. 
\end{proof}

\bibliographystyle{alpha}
\bibliography{Bibliography.bib} 

\end{document}